\documentclass[authoryear]{elsarticle}

\nonstopmode

\usepackage[utf8]{inputenc}
\usepackage{graphicx}
\usepackage{color}
\usepackage{verbatim}
\usepackage{hyperref}
\usepackage{url}
\usepackage{hypernat}

\usepackage{subfigure}

\usepackage{amsmath}
\usepackage{amsthm}
\usepackage{amsfonts}
\usepackage{amssymb}
\newtheorem{thm}{Theorem}

\usepackage{stmaryrd}

\usepackage{tikz}
\usetikzlibrary{positioning}
\usetikzlibrary{calc}
\usetikzlibrary{chains}
\usetikzlibrary{shapes}

\usepackage[margin=10pt,font=small,labelfont=bf,labelsep=period]{caption}

\newcommand\cC{\mathcal{C}}
\newcommand\bbR{\mathbb R}
\newcommand{\doi}[1]{\href{http://dx.doi.org/#1}{doi: #1}}

\usepackage{geometry}
\def\marklegendre#1{
    \foreach \i / \ln in {
       0/0.02544604,  1/0.12923441, 2/0.29707742, 3/0.5,
       4/0.70292258, 5/0.87076559, 6/0.97455396}
    { coordinate [pos=\ln] (#1\i) }
  }

\long\def\algcmt#1{\emph{\{ #1 \}%
}}

\pgfdeclarelayer{background}
\pgfdeclarelayer{grid}
\pgfdeclarelayer{foreground}
\pgfsetlayers{background,grid,main,foreground}

\def\intd{\, d}

\let\cite=\citep

\def\dlmf#1{\href{http://dlmf.nist.gov/#1}{#1}}

\begin{document}


\begin{frontmatter}
\title{Quadrature by Expansion: A New Method for the Evaluation of Layer~Potentials}

\author[cims]{Andreas Klöckner}
\ead{kloeckner@cims.nyu.edu}
\author[dartmouth]{Alexander Barnett}
\author[cims]{Leslie Greengard}
\author[cims]{Michael O'Neil}
\address[cims]{Courant Institute of Mathematical Sciences, 251 Mercer
Street, New York, NY 10012}
\address[dartmouth]{Department of Mathematics,
  Dartmouth College, Hanover, NH, 03755}

\begin{abstract}
  Integral equation methods for the solution of partial differential equations,
  when coupled with suitable fast algorithms, yield geometrically flexible,
  asymptotically optimal and well-conditioned schemes in either interior or exterior
  domains. The practical application of these methods, however, requires the
  accurate evaluation of boundary integrals with singular, weakly singular or nearly
  singular kernels.  Historically, these issues
  have been handled either by low-order product integration rules
  (computed semi-analytically), by singularity subtraction/cancellation,
  by kernel regularization and asymptotic analysis, or by the construction of
  special purpose ``generalized Gaussian quadrature'' rules.
  In this paper, we present a systematic, high-order
  approach that works for any singularity (including hypersingular
  kernels), based only on the assumption that the field induced by the
  integral operator is locally smooth when restricted to either the
  interior or the exterior.  Discontinuities in the field across the
  boundary are permitted. The scheme, denoted QBX (quadrature by
  expansion), is easy to implement and compatible with fast
  hierarchical algorithms such as the fast multipole method.
  We include accuracy tests for a variety of integral operators in two
  dimensions on smooth and corner domains.
\end{abstract}

\begin{keyword}
Layer Potentials\sep Singular Integrals\sep Quadrature\sep High-order accuracy\sep
Integral equations
\end{keyword}

\end{frontmatter}


\section{Introduction}
\label{sec:intro}
One of the difficulties encountered in the practical application of
integral equation methods lies in the need to evaluate integrals
with singular or weakly singular kernels in complicated domains. For the sake of
concreteness, we assume the computational task is to compute
layer potentials such as the single and double layer potentials
\begin{align}
  S\sigma(x) := \int_\Gamma G(x,x')\sigma(x') \intd x',
  \label{eq:kernelint}   \\
  D\mu(x) := \int_\Gamma \frac{\partial G}{\partial \hat n_{x'}}(x,x')\mu(x') \intd x'
  \label{eq:kernelintd}
\end{align}
for target points $x$ on a closed, smooth contour
$\Gamma\subset \mathbb R^2$, where $G$ is the Green's function for an
underlying elliptic PDE and $\hat n_{x'}$ denotes the outward unit
normal at $x'$.  In the case of the double layer potential,
it is typically the principal value of $D\mu$ that is desired for $x \in \Gamma$.

In the present paper, we will restrict our attention to the Helmholtz equation
\[ \Delta \phi + k^2 \phi = 0, \]
for which
\begin{equation}
  G(x,x') = \frac i4 H^{(1)}_0(k|x-x'|),
  \label{eq:helmholtz-ker}
\end{equation}
where $H^{(1)}_0$ denotes the Hankel function of the first kind of order 0. $H^{(1)}_0$
satisfies the Sommerfeld radiation condition
\[
  \lim_{r\to\infty}
  r^{1/2}
  \left(\frac{\partial}{\partial r} - ik \right) H^{(1)}_0
  = 0,
\]
where $r=|x-x'|$ and $k\in\mathbb C$ with $\operatorname{Im} k \ge 0$.
Using this Green's function,
it is well-known that the operators in \eqref{eq:kernelint}, \eqref{eq:kernelintd}
are both weakly singular
when acting on the boundary, and the quantities of interest
are well-defined improper integrals. Off the boundary,
the double layer potential must be treated with more care as the singularity
is of the order $1/|x-x'|$ and the limiting value has a jump of
$\mu(x')$ at the point $x' \in \Gamma$. Operators involving stronger
singularities, such as the gradient of $S\sigma$ or $D\mu$,
are often also of practical interest.

When the target $x$ is far from the boundary, the integrands in
\eqref{eq:kernelint}, \eqref{eq:kernelintd} are smooth, and high-order
quadratures can be obtained by standard methods. Difficulties are encountered
only when $x$ is either on or near the boundary.
The problem of quadrature for singular or nearly singular integrals,
of course, has a rich history and, we do not seek to review the
literature here (see, for example, the texts
\citep{atkinson_1997,brebbia_1984,kress_1999}).  The most common approach is
probably product integration, that is to say exact integration of the kernel
multiplied by a piecewise polynomial approximation of the density $\sigma$ or
$\mu$ on a piecewise smooth approximation of the boundary. Purely analytic
rules, however, tend to be limited to a few singularities (such as $\log
|x-x'|$ in 2D or $1/|x-x'|$ in 3D) and low order approximations of the density
and boundary.  To handle the kernels $H^{(1)}_0(k|x-x'|)$ or $e^{ik|x-x'|}/|x-x'|$,
analytic rules are often combined with numerical quadratures through the method
of {\em singularity subtraction}. More precisely, it is easy to verify that
\[ \frac i4 H^{(1)}_0(k|x-x'|) - \frac{1}{2\pi} \log|x-x'|  \quad\text{and}\quad
 \frac{e^{ik|x-x'|}}{|x-x'|} - \frac{1}{|x-x'|}
\]
are smoother functions than the original kernels themselves and somewhat
easier to integrate numerically.
More generally, if the kernel $G_1(x,x')$ can be integrated, say, on a flat surface by
analytic means, then integrating $G_2(x,x')-G_1(x,x')$ is an easier task
if $G_2$ and $G_1$ have the same leading order singularity
\citep{davis_1984,farina_2001,johnson_1989}.

Three other powerful approaches are (a) to design special purpose quadratures
that integrate a specific class of singular functions with high-order accuracy
\citep{alpert_hybrid_1999,bremer_nonlinear_2010,kapur_high-order_1997,
kolm_numerical_2001,kress_numerical_1995,
sidi_1988,strain_1995,yarvin_generalized_1998,helsing_integral_2009},
(b) to find a change of variables that removes the principal singularity
\citep{bruno_2001, davis_1984, duffy_1982, graglia_2008, hackbusch_sauter_1994, jarvenpaa_2003,
khayat_2005, kress_boundary_1991, schwab_1992, ying_2006}, and
(c) to regularize the kernel so that smooth rules can be applied,
followed by corrections through asymptotic analysis or Richardson extrapolation
\citep{beale_lai_2001, goodman_1990, haroldson_1998, lowengrub_1993, schwab_1992}.
By contrast with singularity subtraction, methods of type (b) are sometimes
referred to as using {\em singularity cancellation}.

In the complex analytic (or harmonic) case, some
remarkable methods have been developed by \citet{helsing_2008a}
for off-surface evaluation. It should be noted that in the two-dimensional case,
several of these alternatives provide extremely effective schemes, especially
\cite{bremer_nonlinear_2010,helsing_integral_2009,helsing_tutorial_2012,helsing_2008a} since they all permit local adaptivity and high order accuracy.

Many applications involve layer potentials which are only defined as principal
value and finite-part (hypersingular) integrals. Especially the latter
are notoriously difficult using classical quadrature
schemes, although methods using integration by parts can be effective
\cite{chapko_numerical_2000}.

The main purpose of the present paper is to introduce a rather
different approach to the evaluation of layer potentials, based on the
fact that the fields $S\sigma$ or $D\mu$ in
\eqref{eq:kernelint}, \eqref{eq:kernelintd} are locally smooth
functions when restricted to either the interior or the exterior,
although they may be discontinuous across the boundary. The scheme,
denoted QBX (quadrature by expansion), is easy to implement, high
order accurate, and requires only a smooth underlying quadrature
scheme. This underlying smooth rule may be global or composite/panel-based
and adaptive. The method is also compatible with fast hierarchical algorithms. QBX
is an extension of the work of
\citet{barnett_evaluation_2012}, who addresses the near but off-surface
evaluation problem.
Unlike the schemes discussed above, it is essentially
dimension-independent, although our numerical experiments
here are limited to the two-dimensional case.

Our approach is somewhat related to the algorithms
discussed in
 \cite{delves_numerical_1967,lyness_numerical_1967}. In those papers, use is
also made of the fact that the function induced by a boundary integral
is smooth in the interior and can be expanded as Taylor series
about an interior point. The viewpoint taken, however, is global and
restricted to the case of Cauchy integrals with analytic data.
In essence, they make use of a
a single expansion center with a radius of convergence
determined by the nearest singularity of the underlying
analytic function itself.
By contrast, we make no
assumptions about the location of the nearest singularity. Instead, we use
expansions centered at
points very close to the boundary $\Gamma,$ and make use of
error estimates that
depend only on the local smoothness of the data and boundary, which may only be
of finite differentiability.

The paper is organized as follows: In Section \ref{sec:idea}, we show
how QBX follows naturally from considerations of potential theory.
Section \ref{sec:math-details} describes the mathematical foundations
of the method, and Section \ref{sec:numerical} demonstrates its
numerical performance. A simple but complete description of the
algorithm can be found in Section \ref{sec:at-a-glance}. Finally, we
discuss additional details and potential extensions of the present
work in Section \ref{sec:details}.
\section{Smooth, high-order quadrature}
\label{sec:idea}
\begin{figure}
  \hfill
  \subfigure[Error in the potential using the trapezoidal rule with
  50 quadrature points.]{
    \label{fig:strip-of-death-coarse}
    \includegraphics[height=5cm]{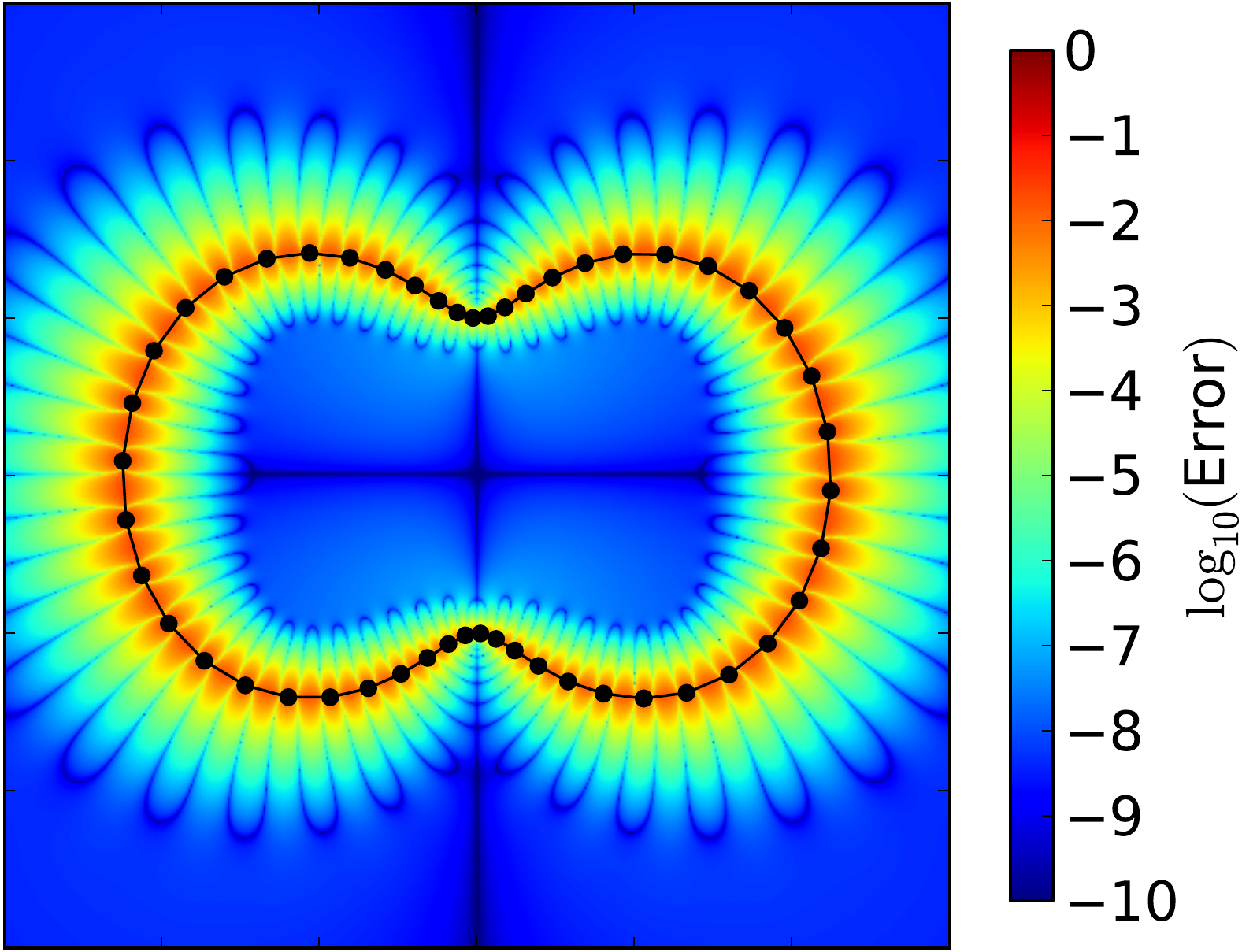}
  }
  \hfill
  \subfigure[Error in the potential using the trapezoidal rule with
  100 quadrature points.]{
    \label{fig:strip-of-death-fine}
    \includegraphics[height=5cm]{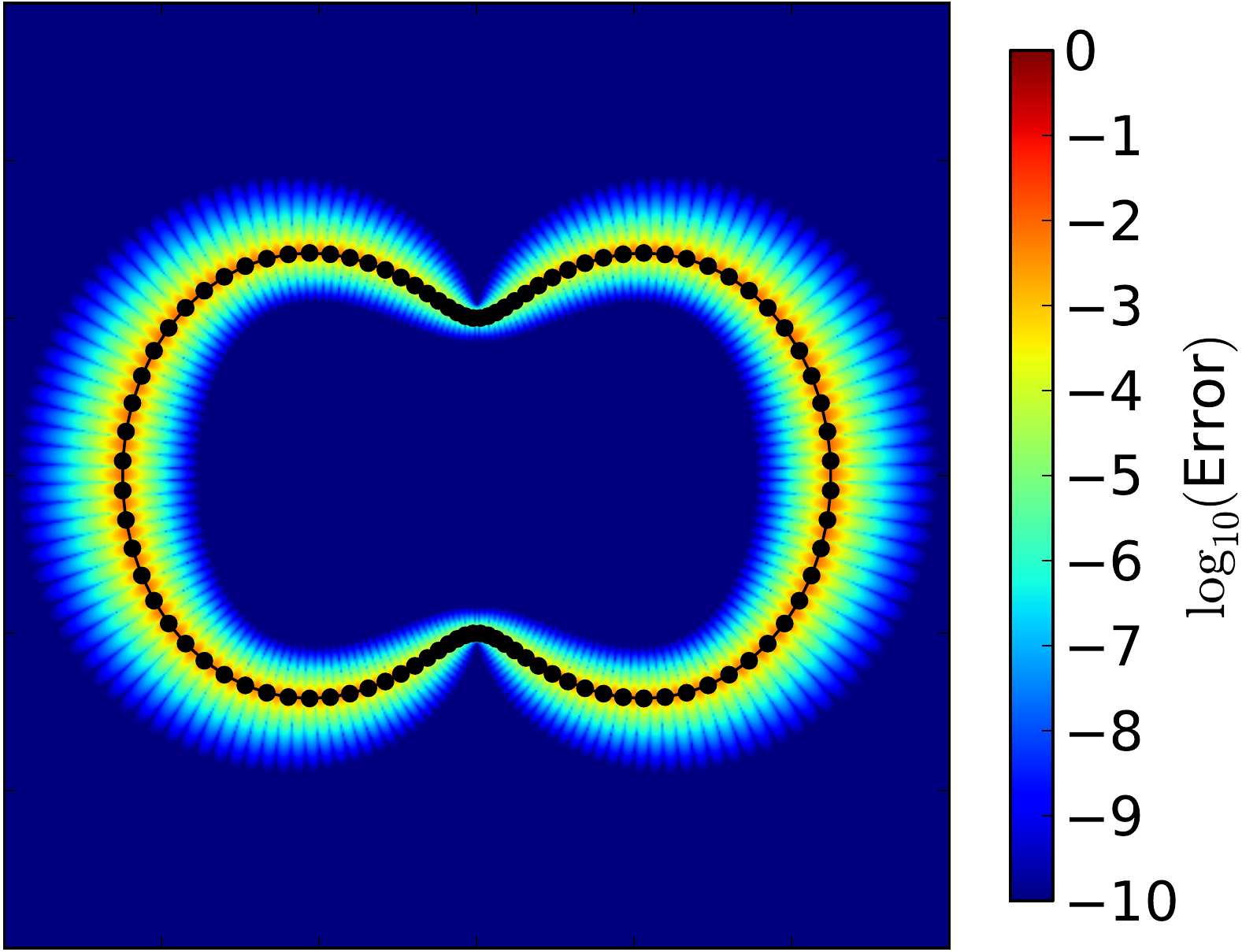}
  }
  \hfill
  \caption{The potential $S\sigma$ is computed using the trapezoidal rule,
a simple, high-order quadrature for smooth functions.}
\label{fig:strip-of-death}
\end{figure}

Let us assume, for the moment, that we are
given a smooth, simply connected closed curve $\Gamma\subset \mathbb R^2$, with a
parametrization $\Gamma = \{ \gamma(t) : 0 \leq t < L \}$.
We denote the interior of $\Gamma$ by $\Omega^-$ and its exterior by $\Omega^+$.
We also assume that we have at our disposal an underlying quadrature
rule capable of integrating smooth (non-singular) functions on $\Gamma$ to high precision.
In two dimensions, one option is the trapezoidal rule, since it is well-known
to achieve superalgebraic convergence for smooth data on closed curves
\citep{davis_1984}.

A very natural question at this point is the following: for a target
location $x$ \emph{away} from $\Gamma$, how well does the trapezoidal rule compute
$S\sigma(x)$ or $D\mu(x)$? Certainly, the integrands in
\eqref{eq:kernelint} and \eqref{eq:kernelintd} are not actually singular in this situation,
so the real question is how close $x$ can be to $\Gamma$ before accuracy is lost.
Before analyzing this error more carefully, let us carry out a simple
computational experiment for the curve
\begin{equation*}
  \gamma(t)=\begin{pmatrix}
    \frac34 \cos(t-\pi/4)(1+\sin(2t)/2)\\
    \sin(t-\pi/4)(1+\sin(2t)/2)
  \end{pmatrix},
\end{equation*}
with $0 \leq t < 2\pi$ for a Helmholtz parameter $k=0.5$.
Using either 50 nodes (Fig. \ref{fig:strip-of-death-coarse}) or
100 nodes (Fig. \ref{fig:strip-of-death-fine}),
we plot the error in both the interior and exterior of $\Gamma$.
The colors in these figures
indicate the absolute value of the pointwise error for the Helmholtz
single-layer potential $S\sigma$ with $\sigma\equiv 1$.

These figures clearly suggest that the region in which the layer potential
is inaccurate shrinks more or less in proportion to the grid spacing $h$ (a fact well-known
to practitioners of potential theory).
To be a little more precise, let
$T_N(S\sigma)$ denote the trapezoidal approximation of $S\sigma$ using $N$ points,
and let
\[ E(x) = |S\sigma(x) - T_N(S\sigma)(x)|.\]
For a fixed $\epsilon$, we define the ``high-accuracy'' region of the plane
as the subset of $\mathbb R^2$ where $E(x) < \epsilon$. This will, in essence,
be all of $\mathbb R^2$ with a neighborhood of $\Gamma$ removed.
The extent of this region depends on both $h$ and $\epsilon$
\citep{barnett_evaluation_2012}.

The fact that $h$-refinement shrinks the
region of inaccuracy, of course, is of no great value in evaluating
layer potentials at points $x$ on the curve $\Gamma$ itself.
For this, let us instead choose a point $c$ off the surface with
\[ c = x + 5h \hat n_{x}\, ,  \]
where $\hat n_{x}$ is the unit normal to $\Gamma$ at
$x$. From our initial experiment, it is reasonable to expect that $c$
is in the ``high-accuracy" region. Assuming $S\sigma$ is a smooth function
in either the interior $\Omega^-$ or the exterior $\Omega^+$, it is easy to see that
\begin{equation}
|T_N(S\sigma)(c)-S\sigma(x)| =
|(T_N(S\sigma)(c)- S\sigma(c)) + (S\sigma(c) - S\sigma(x))| =
O(\epsilon + h)
\label{eq:first-order-estimate}
\end{equation}
since $|c-x| = O(h)$, under the assumption that
the trapezoidal rule is accurate to precision $\epsilon$.
In other words, the approximate value $T_N(S\sigma)(c)$ is a
first-order accurate approximation of the on-surface value
$S\sigma(x)$, within the error $\epsilon$.

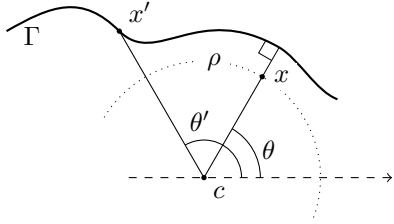
\begin{figure}
  \centering
  \begin{tikzpicture}
    \coordinate (c) at (0,0) ;
    \path (c) ++(45+75:2.25) coordinate (s);
    \path (c) ++(45+15:1.55) coordinate (t);
    \path (c) ++(45+15:2) coordinate (t-to-curve);

    \path (s) ++(-1.5,0) coordinate (curve-before);
    \path (t) ++(1,-0.3) coordinate (curve-after);
    \draw [thick]
      (curve-before)
      ..controls +(30:0.7) and +(180-45:0.7) ..
      (s)
      node [pos=0.2,anchor=north] {$\Gamma$}
      ..controls +(-45:0.7) and +(45+90+15:1.25) ..
      (t-to-curve)
      ..controls +(45+270+15:0.3) and +(160:0.3) ..
      (curve-after) ;

    \fill (c) circle (1pt);
    \fill (t) circle (1pt);
    \fill (s) circle (1pt);

    \draw [dotted] ++(-20:1.55) arc (-20:45+105:1.55);

    \node at (85:1.55) [fill=white] {$\rho$};
    \draw (c) -- (t) ;
    \draw (c) -- (s) ;
    \node at (c) [anchor=north west] {$c$};
    \node at (s) [anchor=south west] {$x'$};
    \node at (t) [anchor=west,fill=white,xshift=0.5mm,inner sep=1mm] {$x$};
    \draw [dashed,->] (c) ++(-1, 0) -- ++(3.5,0);

    \draw (c) ++(0.5,0) arc (0:45+75:0.5);
    \path (c) ++(45*0.8+75*0.8:0.5)
      node [anchor=south] {$\theta'$};

    \draw (c) ++(0.75,0) arc (0:45+15:0.75) ;
    \path (c) ++(45/2+15/2:0.75)
      node [anchor=west] {$\theta$};

    \draw [very thin] (t) -- (t-to-curve);
    \draw [very thin]
      let
        \p1 = ($ 0.1*(t-to-curve) - 0.1*(c) $),
        \p2 = (-\y1,\x1)
      in
      ($(t-to-curve)!0.1!(c)$) -- ++(\p2) -- ++(\p1) ;

  \end{tikzpicture}
  \caption{
    Geometric situation of Graf's addition theorem
    with sources along the curve $\Gamma$, as
    used in \eqref{eq:local-exp} and
    \eqref{eq:graf-coefficient}. Note that $x$ will
    reside on $\Gamma$ further on in the discussion.
  }
  \label{fig:graf-addition}
\end{figure}

Remarkably, it is straightforward to improve matters even further.
Instead of evaluating $S\sigma(c)$, let us \emph{expand} $S\sigma$
about $c$ to order $p$.  The classical separation-of-variables
representation of a smooth solution to the homogeneous Helmholtz
equation in a disk centered at $c$ takes the form
\begin{equation}
 \phi(x) = \sum_{l=-\infty}^\infty \alpha_l J_l(k \rho) e^{-i l \theta}
 \label{eq:local-exp}
\end{equation}
where $(\rho,\theta)$ denote the polar coordinates of the target $x$
with respect to the expansion center $c$, and $J_l$ is the Bessel
function of order $l$ (see Fig. \ref{fig:graf-addition}).  For the
single layer potential $S\sigma$, the coefficients $\alpha_l$ in the
expansion (\ref{eq:local-exp}) can be computed analytically:
\begin{equation}
  \alpha_l = \frac{i}{4} \,
  \int_\Gamma H^{(1)}_l(k|x'-c|) e^{i l \theta'} \sigma(x') \intd x',
  \qquad
  (l = -p, -p+1, \ldots, p)
  \label{eq:graf-coefficient}
\end{equation}
where $(|x'-c|,\theta')$ denote the polar coordinates of the point $x'$ with respect
to $c$.
These formulas follow immediately from
Graf's addition theorem \citep[][\dlmf{10.23.7}]{olver_nist_2010},
\begin{equation}
  H^{(1)}_0(k|x-x'|) = \sum_{l=-\infty}^\infty H^{(1)}_l(k|x'-c|) e^{i l \theta'}
  J_l(k |x-c|) e^{-i l \theta},
  \label{eq:graf-addition}
\end{equation}
by interchanging the order of summation and integration.
We note that Graf's addition theorem is generally applicable only if
the target $x$ is closer to the center than the source $x'$:
\begin{equation}
|x-c|<|x'-c|.
\label{eq:x-c-closest}
\end{equation}

\begin{figure}
  \hfill
  \subfigure[$p=3$, $N=80$ quadrature nodes]{
    \label{fig:expansions-low-order}
    \includegraphics[height=5.2cm]{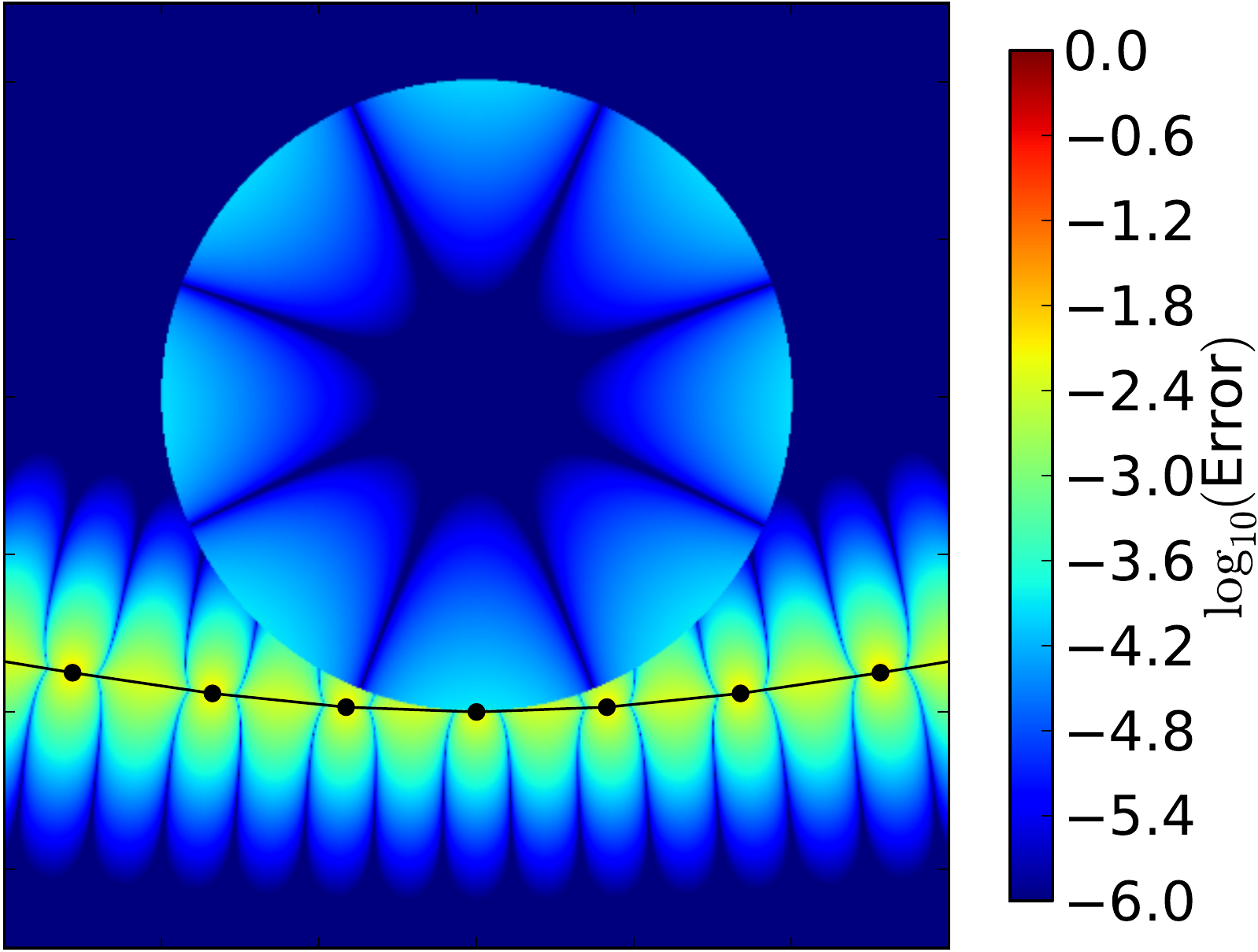}
  }
  \hfill
  \subfigure[$p=6$, $N=80$ quadrature nodes]{
    \label{fig:expansions-med-order}
    \includegraphics[height=5.2cm]{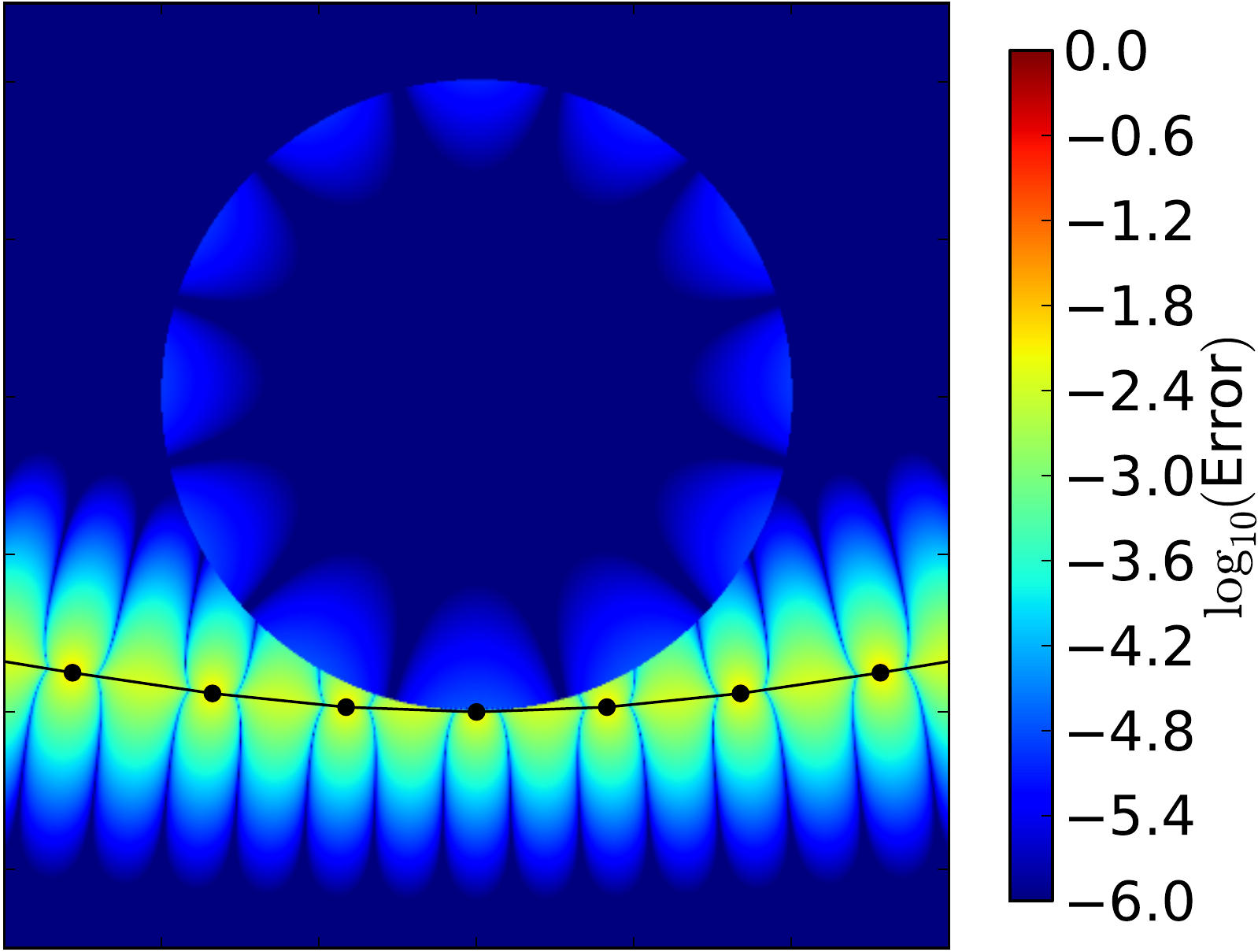}
  }
  \hfill

  \hfill
  \subfigure[$p=12$, $N=80$ quadrature nodes]{
    \label{fig:expansions-high-order-bad}
    \includegraphics[height=5.2cm]{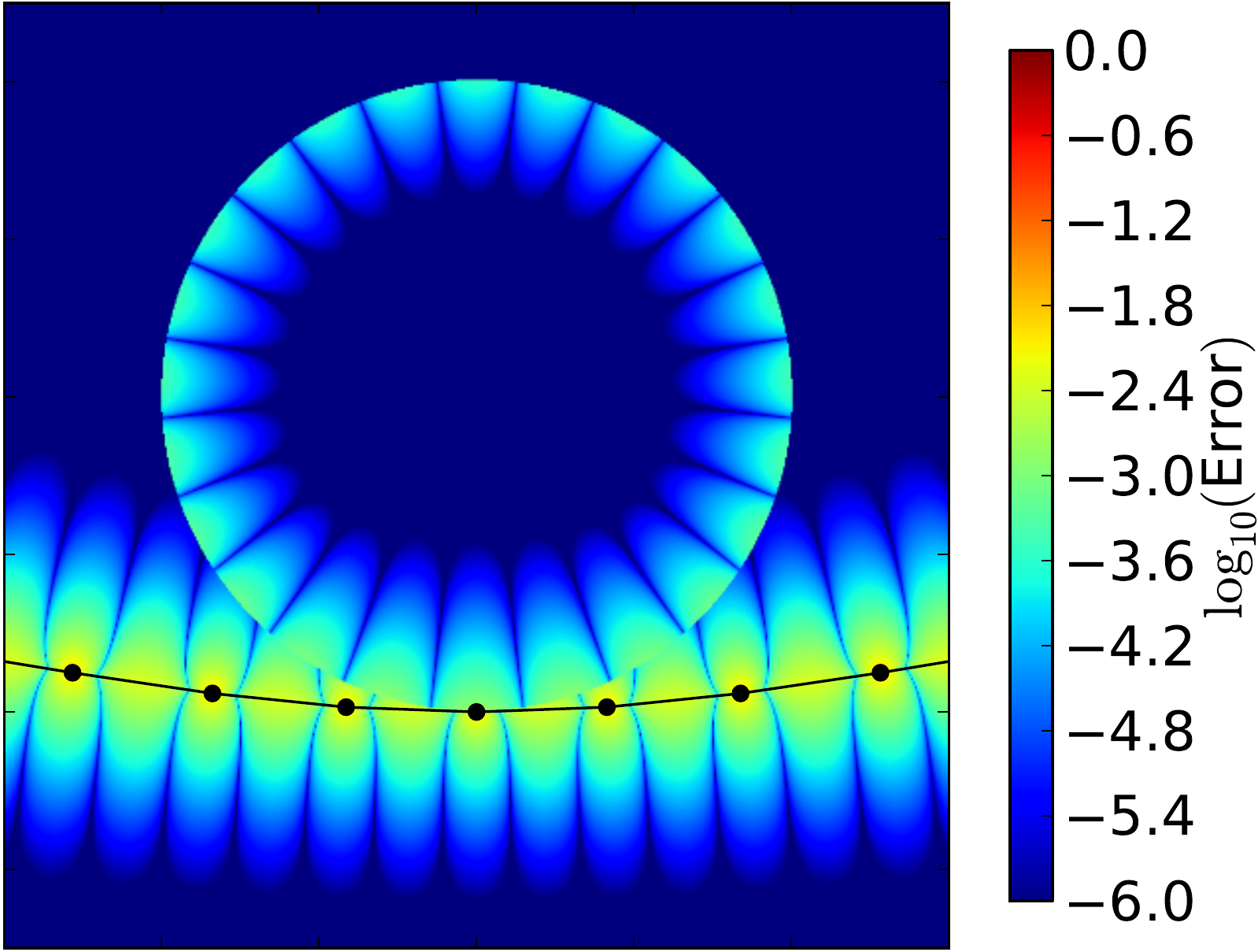}
  }
  \hfill
  \subfigure[$p=12$, $N=240$ quadrature nodes]{
    \label{fig:expansions-high-order-good}
    \includegraphics[height=5.2cm]{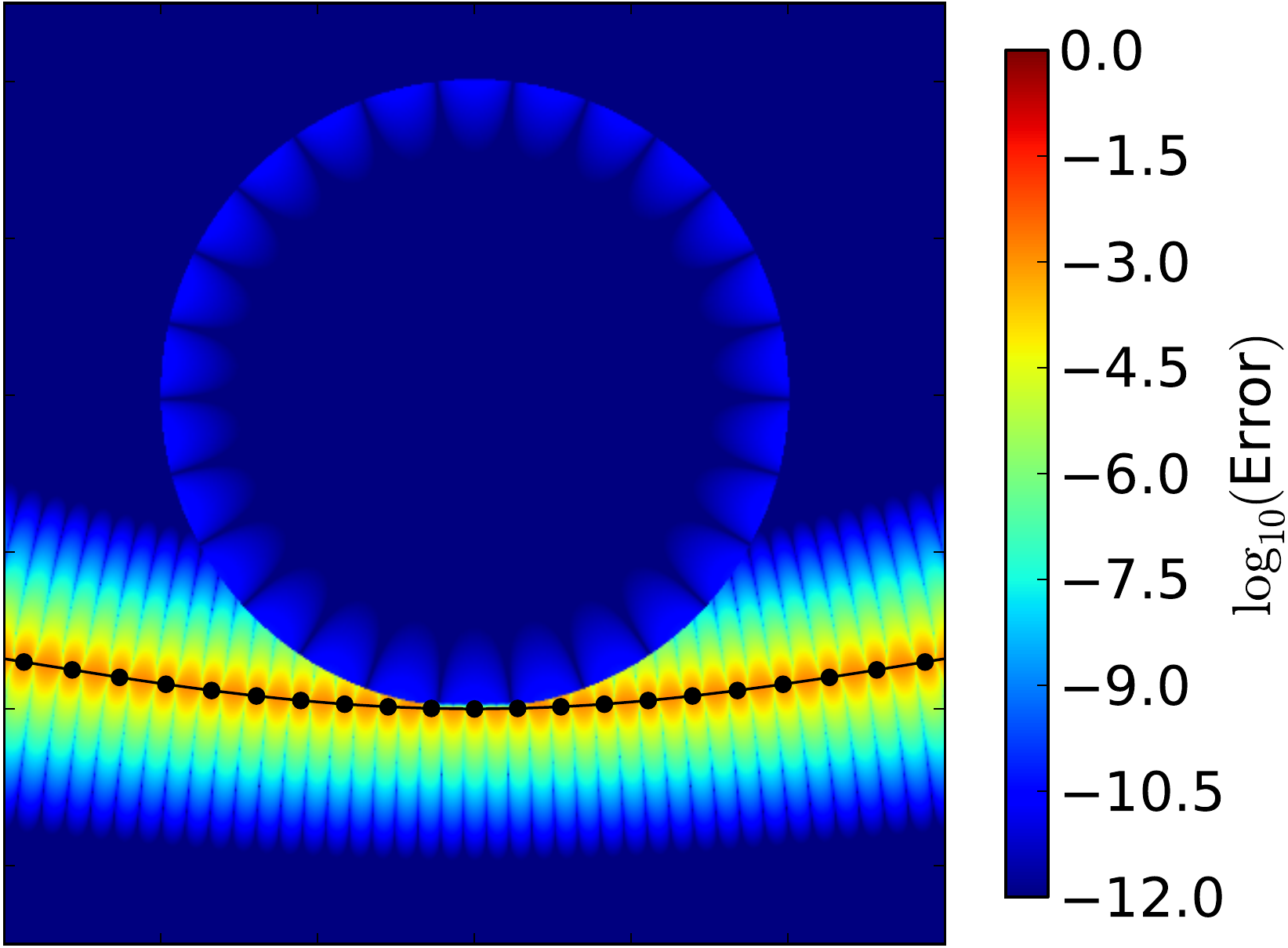}
  }
  \hfill
  \caption{The potential $S\sigma$ is computed using the trapezoidal rule
$\Gamma$ with either $N = 80$ or $N=240$ points, except in a disk of radius
$|c - x|$ centered at an off-surface point $c$ that lies in the
``high-accuracy" region of the trapezoidal rule (here, approximately
$3h$ away from the curve in a), b), and c)). Only a portion of
the boundary $\Gamma$ is plotted, and $x$ is the point
where the disk and $\Gamma$ are tangent.
We plot the error in the disk using various expansion orders $p$ and
numbers of quadrature nodes.}
\label{fig:accuracy-story}
\end{figure}

The integral defining $\alpha_l$ is similar to that
defining the original layer potential, except that $H^{(1)}_0(k|x'-c|)$ has
been replaced with the more complicated but still smooth function
$H^{(1)}_l(k|x'-c|) e^{i l \theta'}$.
Because of this smoothness, we evaluate $\alpha_l$ using the same trapezoidal rule $T_N(\alpha_l)$.
When seeking to evaluate $S\sigma(x)$, however, we evaluate the local expansion
(\ref{eq:local-exp}) instead.
Figure \ref{fig:expansions-low-order} shows the result of using an expansion
of order $p=3$ superimposed on the naive use of the trapezoidal rule $T_N$ to
compute the single layer potential directly. The circular `cut-out'
regions in this and the following figures indicate where local
expansions were used to approximate $S\sigma$. Note that
the error in $S\sigma$ computed by
the local expansion is \emph{far} smaller than the error of the naive
computation throughout the circular region.
In effect, the local expansion allows us to
punch a disk-shaped hole into the region of inaccuracy.
This is precisely the idea underlying the close evaluation scheme
of \citet{barnett_evaluation_2012} for targets near, but not on, $\Gamma$.

In this paper, we take the approach one step further.
Namely, we investigate the use of the expansion
\eqref{eq:local-exp} in evaluating the layer potential $S\sigma(x)$ for
$x$ that actually lie on the curve $\Gamma$.

Formally, it is worth noting that the radius of convergence of the
local expansion about $c$ is $r = \min_{x'\in\Gamma} |x'-c|$ as Graf's
addition theorem requires it. Thus, we are seeking to evaluate a local expansion
{\em at} its radius of convergence where the accuracy is most difficult to
analyze. This difficulty, however, stems from the use of the addition theorem
for a singular field (the potential due to a point source $H^{(1)}_0$).
Figure \ref{fig:expansions-low-order} shows
that the expansion is, in fact, accurate:
it provides about four digits of precision uniformly.
Loosely speaking, accuracy follows from the fact that the field induced by the
layer potential is (one-sided) smooth in the interior and exterior domains $\Omega^-$ or $\Omega^+$.
The analytic issues here are somewhat involved and concern estimates
on the decay of the coefficients $\alpha_l$ in terms of the smoothness of the curve
$\gamma(t)$ and the density $\sigma(t)$. Those estimates are established
in \citep{epstein_convergence_2012}, and we will invoke them, as needed, below.

If using an expansion of order $p=3$ provides an accurate value for $S\sigma(x)$,
is it perhaps possible to obtain even more
accuracy by further increasing $p$? Figure
\ref{fig:expansions-med-order} shows the results of such an experiment.
By setting $p=6$, the accuracy of the potential in the vicinity of (and really also \emph{on} $\Gamma$)
increases from four to six digits.
Further increasing the order to $p=12$, however, causes a significant
{\em loss} of accuracy (Fig. \ref{fig:expansions-high-order-bad}).
A consideration of the integrand in
\eqref{eq:graf-coefficient} shows why this occurs. As
$p$ is increased, both factors in the integrand $H^{(1)}_l(k|x'-c|) e^{i l \theta'}$
increase in complexity: $e^{i l \theta'}$ by becoming more oscillatory,
and $H^{(1)}_l(k|x'-c|)$ by becoming more sharply peaked.
This combined effect leads to the resolution of the underlying trapezoidal rule
being exceeded.  Thus, in the experiment of Figure
\ref{fig:expansions-high-order-bad}, the coefficients $\alpha_l$ are both large
and wrong.  Fortunately, once identified, this issue is easy to resolve.
Indeed, simply increasing the number of points in the trapezoidal rule
compensates for the added complexity of the integrands involved in computing the higher
order coefficients, and with $p=12$ more than ten digits of accuracy are
achieved (Fig. \ref{fig:expansions-high-order-good}).

The sequence of experiments described thus far suggest a path to
the high-order accurate evaluation of layer potentials as operators on the boundary.
It also highlights one aspect of the scheme
that requires careful analysis, namely the interplay between
$h=1/N$ and $p$.
The local grid spacing $h$ must be chosen small enough so that the coefficients in the local expansion
\eqref{eq:local-exp}
are computed
with the necessary precision. We have concentrated in our experiments on a single
boundary point $x$. To evaluate
$S\sigma$ everywhere on the boundary, we will simply introduce a
large number of off-surface expansions centers whose corresponding disks cover a
neighborhood of $\Gamma$. We choose our expansion centers so that
any desired target point is in the interior or on the boundary of one of these disks,
enabling the application of QBX. If a target point is not in any of these disks,
by definition it will be in the ``high-accuracy'' region
associated with the original trapezoidal approximation.
In the simplest approach, one may introduce an expansion center for each
discretization node on the boundary.
The procedure applied above to the single layer potential can be used just as well
for the evaluation of integrals with hypersingular kernels.

None of the observations made above
change substantially if we replace the trapezoidal rule with another
high-order quadrature.
Figure \ref{fig:legendre-strip-of-death} presents the
analog of Figure \ref{fig:strip-of-death} for a circle discretized
using composite Gauss-Legendre quadrature.
Using ideas from \citet{barnett_evaluation_2012}, we believe that
the error contours (ignoring nodal oscillations) are the
conformal images of the Bernstein ellipses \citep{davis_1984}
associated with the integrand on each panel.

\begin{figure}
  \hfill
  \subfigure[Error in potential from (smooth) composite Gauss-Legendre
    quadrature, with 5 panels consisting of 10 quadrature nodes each.]{
    \includegraphics[height=5cm]{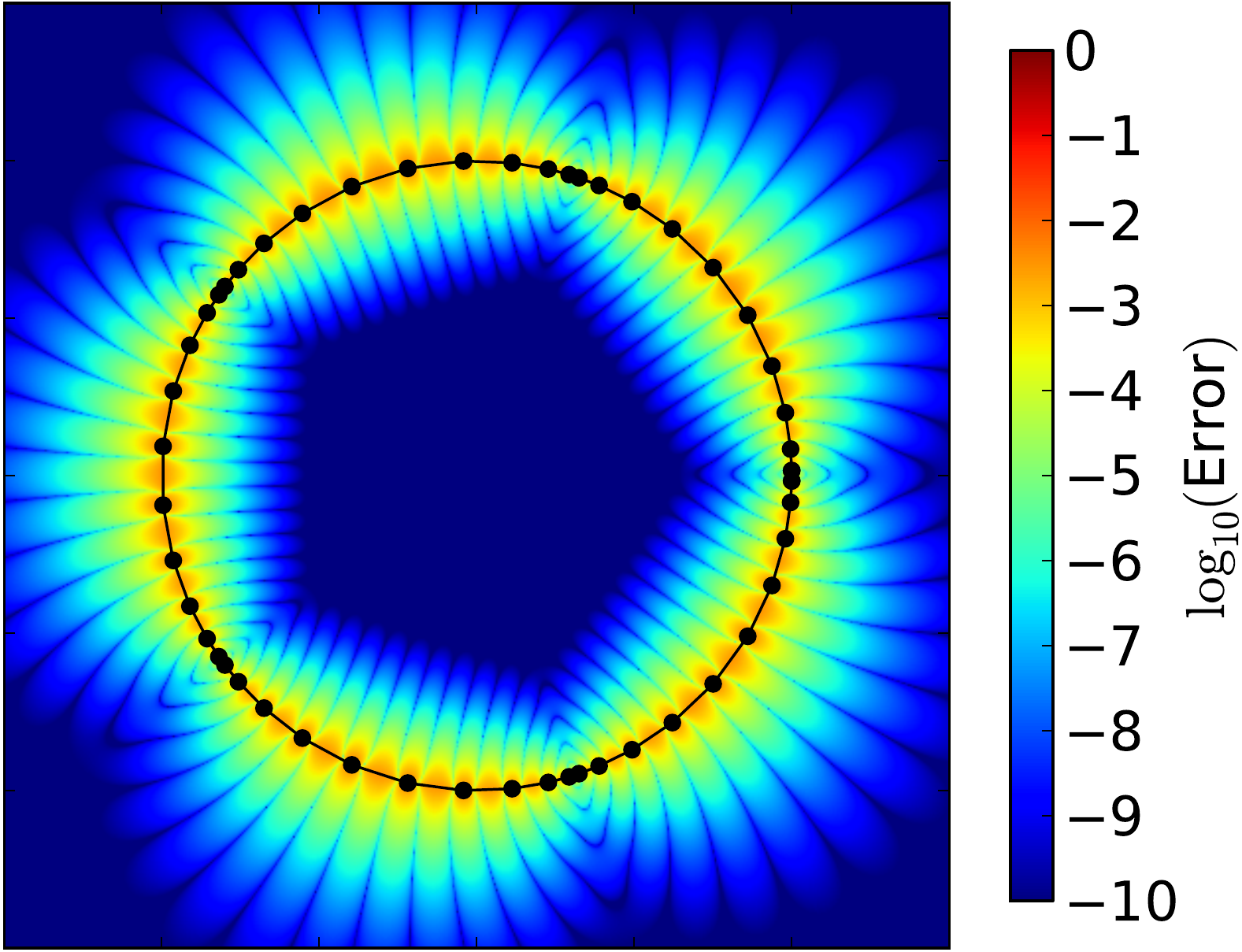}
  }
  \hfill
  \subfigure[Error in potential from (smooth) composite Gauss-Legendre
    quadrature, with
    10 panels consisting of 10 quadrature nodes each.]{
    \includegraphics[height=5cm]{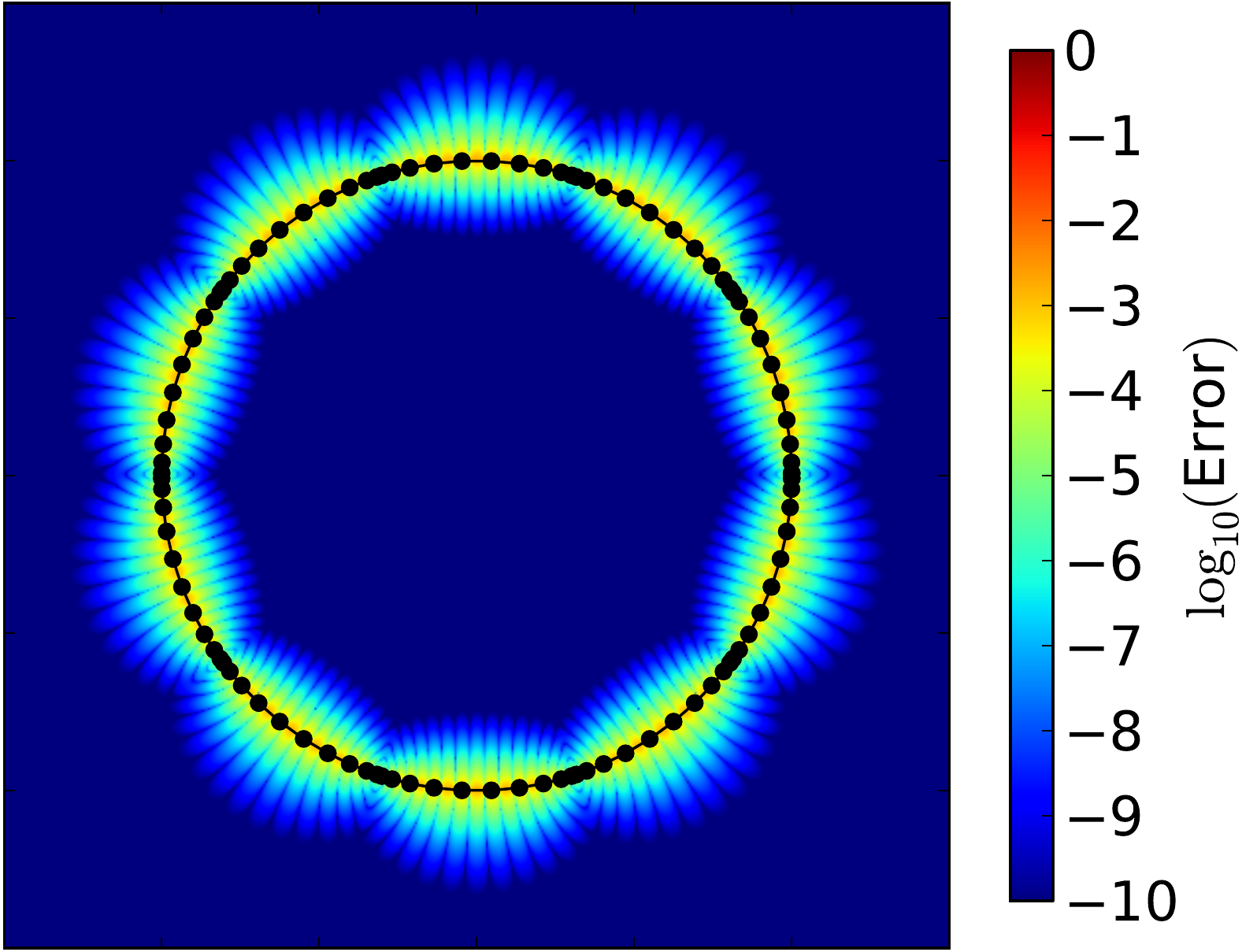}
  }
  \hfill
  \caption{The potential $S\sigma$ computed using composite 10$th$ order
    Gauss-Legendre quadrature.
  }
  \label{fig:legendre-strip-of-death}
\end{figure}

\subsection*{QBX as a regularization scheme}
\label{sec:regularization}
For readers familiar with multipole/partial wave expansions, the
numerical results above may come as a surprise. After all, given a finite
set of quadrature nodes (as in our computational examples), if the
order of the local expansion were sufficiently high, it should converge
to the field induced by a finite set of singular sources, namely the
quadrature nodes. Instead, it is reproducing the continuous layer
potential, even at the quadrature nodes themselves. There are two
interpretations of this fact.

For some, it is most natural to understand this in terms of series
approximations of smooth functions, as introduced above.  For others, it is
perhaps useful to interchange the order of summation and integration and write
\[
  S\sigma(x) \approx \int_\Gamma
  G_p(x,x') \sigma(x') \intd x',
\]
where
\begin{equation}
  G_p(x,x') = \sum_{l=-p}^p
  H^{(1)}_l(k|x'-c|) e^{i l\theta'} J_l(k |x-c|) e^{-i l\theta},
  \label{eq:qbx-kernel}
\end{equation}
for a target $x$ on or near the boundary.  That is, we can interpret
the entire procedure as substituting the original Green's function
$G$ with $G_p$. It turns out that $G_p$ is a surprisingly good filter.
It regularizes the kernel in such a way that high-order accuracy is
achieved without the need for additional correction. From this
perspective, the need to decrease $h$ with $p$ is due to the fact that
$G_p$ itself is a more and more sharply peaked integrand as $p$
increases.

\section{Mathematical Foundations of QBX}
\label{sec:math-details}
\subsection{Error analysis}
\label{sec:analysis}
We turn now to the
principal result justifying the use of QBX as a quadrature scheme.
We restrict our attention to composite Gauss-Legendre quadrature, but
the proof is analogous for any smooth high-order rule.
In what follows, $h$ will be used to denote the \emph{panel size}
in the composite Gauss-Legendre grid, rather than the point spacing
used previously in discussing the trapezoidal rule. We apologize
for this abuse of notation.

\begin{thm} \label{thm:qbx}
Suppose that $\Gamma$ is a smooth, bounded curve
  embedded in $\bbR^2,$ that $B_r(c)$ is the ball of radius
  $r$ about $c$, and that $\overline{B_r(c)} \cap \Gamma = \{x\}$.
Let $\Gamma$ be divided into $M$ panels, each of length
$h$ and let $q$ be a non-negative integer that defines
the number of nodes of the smooth Gaussian quadrature used to compute the coefficients
$\alpha^{\mathrm{QBX}}_l$ according to the formula (\ref{eq:graf-coefficient}).
For $0<\beta<1,$ there are constants $C_{p,\beta}$ and $\tilde C_{p,q,\beta}$
so that if $\sigma$ lies in the Hölder space
$\cC^{p,\beta}(\Gamma)\cap \cC^{2q,\beta}(\Gamma)$, then
\begin{equation}
  \left| S\sigma(x) -
  \sum_{l=-p}^{p}\alpha^{\mathrm{QBX}}_l J_{l}(k|x-c|)e^{-il\theta_{cx}}\right|
  \leq
  \Bigg(
    \underbrace{
      C_{p,\beta} \; r^{p+1}
      \|\sigma\|_{\cC^{p,\beta}(\Gamma)}
    }_{\text{Truncation error}}
    +
    \underbrace{
      \tilde C_{p,q,\beta} \left( \frac{h}{4r} \right)^{2q}
      \|\sigma\|_{\cC^{2q,\beta}(\Gamma)}
    }_{\text{Quadrature error}}
  \Bigg).
  \label{eq:underlying-quad-estimate}
\end{equation}
\end{thm}

\begin{proof}
  We begin by writing
\begin{equation}
E =
  \left| \left(S\sigma(x) -
\sum_{l=-p}^{p}\alpha_l J_{l}(k|x-c|)e^{-il\theta_{cx}} \right) +
\sum_{l=-p}^{p}\left(\alpha_l  - \alpha^{\mathrm{QBX}}_l\right)
  J_{l}(k|x-c|)e^{-il\theta_{cx}}
\right| .
\label{error_anal}
\end{equation}
The first term stems from using a \emph{truncated} $p$-term
expansion in Bessel functions to approximate $S\sigma$, while the second
term is the error that stems from the numerical approximation of the coefficients
in the truncated series. It is shown in \citep{epstein_convergence_2012} that the first error
is of the order $r^{p+1} \, \|\sigma\|_{\cC^{p,\beta}(\Gamma)}$.
For the second term, we note that on a curve segment $\Gamma_i$ of length $h$, the
standard estimate for $q$-point Gauss-Legendre quadrature
is \cite[][(2.7.12)]{davis_1984}
\begin{eqnarray}
\left| \int_{\Gamma_i}
  H^{(1)}_l(k|x'-c|) e^{i l \theta'} \sigma(x') \intd x' -
  \sum_{j=1}^q
  H^{(1)}_l(k|x_j-c|) e^{i l \theta_j} \sigma(x_j) \, w_j \right|
\hspace{2in} \hfill
\nonumber \\
\hspace{1in} \leq
 \frac{h^{2q+1}}{2q+1}\frac{(q!)^4}{(2q)!^3}
 \|
 D^{2q} (H^{(1)}_l(k|x'-c|) e^{i l \theta'} \sigma(x'))
 \|_{\infty,\Gamma_i}
\label{GLestimate}
\end{eqnarray}
where $D^{n}$ denotes the $n^{\rm th}$ derivative of the integrand with respect to the
integration parameter along $\Gamma_i$.
A straightforward combination of Stirling's approximation
\[ \sqrt{2\pi} n^{n+\frac{1}{2}} e^{-n} < n! <
 2 \sqrt{\pi} n^{n+\frac{1}{2}} e^{-n}, \]
summing over all panels, and bounds on the derivative
allow us to write this term as
\begin{equation}
\bigg|\sum_{l=-p}^{p}\alpha_l J_{l}(k|x-c|)e^{-il\theta_{cx}} -
\sum_{l=-p}^{p}\alpha^{\mathrm{QBX}}_l J_{l}(k|x-c|)e^{-il\theta_{cx}} \bigg|
\leq C_{p,q,\beta} \left( \frac{h}{4r} \right)^{2q}.
\label{qbxquad-error}
\end{equation}
Combining the two estimates yields the desired result.
\end{proof}

There are several aspects of the preceding theorem that are worth noting.
\begin{itemize}
\item
The two contributors to the error in the QBX
approximation (\ref{eq:underlying-quad-estimate})
are quite different. By placing an expansion center at a distance
$r = O(h)$ away from the point $x \in \Gamma$, the analytic truncation error
is of order $h^{p+1}$.
This error goes to zero under quadrature mesh refinement.
In order for the second component of the error to be small, however, we need
$\frac{h}{4r} < 1$, so that $r > h/4$. A requirement of this type is essential,
corresponding to the fact that if the expansion center is too close to the
boundary relative to the discretization, accuracy will be lost.
It is perhaps informative to set $r = h/2$,
and write the error $E$ from
(\ref{error_anal}) in the form
\[  E = O(\epsilon + h^{p+1})\, , \]
where $\epsilon = \left( \frac{1}{4} \right)^q$.
Used in this manner, QBX is not classically convergent,
but converges with controlled precision.
\item
If one wants to achieve a classically convergent scheme, it suffices
to refine $r$ more slowly than $h$
(say, with  $r = \sqrt{h}$).
We prefer to keep the error components separate for the sake of clarity and
because it permits additional tests of numerical consistency.
\item
Note that
the estimate (\ref{eq:underlying-quad-estimate})
explains the behavior of QBX discussed in Section \ref{sec:idea},
particularly the results shown in Fig. \ref{fig:expansions-high-order-bad}.
\item
For the sake of simplicity, Theorem \ref{thm:qbx} assumes that the curve
$\Gamma$ is divided into equal-sized segments. In practice, with an adaptive
discretization of the curve, a slightly different version of the result is needed.
Since the difficulty in the error analysis is entirely local and the estimates
are similar to those obtained above, we omit the rather cumbersome analysis.
\item In practice, one is often interested in evaluating
the double-layer potential $D\mu$, or some derivative of $S\sigma$ or $D\mu$.
Straightforward analysis shows that, for $n$ derivatives of the Green's function,
the error estimate in (\ref{eq:underlying-quad-estimate}) is multiplied by a
factor of $r^{-n}$.  The use of QBX
for such calculations is discussed in the next section.
\end{itemize}

\subsection{Derivatives, jumps, and principal value integrals}
\label{sec:derivatives}
Up to this point, we have focused on the calculation of the
single-layer potential $S\sigma$. For the double layer $D\mu$ defined in
\eqref{eq:kernelintd}, the scheme is only slightly different.
The coefficients in \eqref{eq:graf-coefficient} are simply replaced by
\begin{equation}
  \alpha^D_l = \frac{i}{4} \,
  \int_\Gamma \frac{\partial}{\partial \hat n_{x'}} H^{(1)}_l(k|x'-c|) e^{i l \theta'} \mu(x') \intd x'.
  \qquad
  (l = -p, -p+1, \ldots, p)
  \label{eq:graf-coefficient-dlp}
\end{equation}
Because the scheme relies on a local expansion of the potential,
subsequent derivatives of $S\sigma$ or $D\mu$ with respect to the \emph{target} location $x$
are particularly easy to obtain by analytic differentiation of the local (Bessel) expansion.

There is a complication which must be dealt with in evaluating
operators other than the single layer potential (which is only weakly
singular). QBX, by its construction, evaluates the one-sided limit of
a layer potential, with the side determined by the location of the expansion center.
In practice, however, one might want to compute the integral
\begin{align}
  D\mu(x) :=
\int_\Gamma \frac{\partial G}{\partial \hat n_{x'}}(x,x')\mu(x') \intd x' \, ,
  \label{eq:kernelintdpv}
\end{align}
for $x\in \Gamma$. As an operator acting on the boundary, $D$ has an
integrable kernel and $D\mu$ is well-defined \cite{kress_1999}.
$D\mu$ is not,
however, equal to its one-sided limit. Using the superscripts
$+$ and $-$ to denote a point in the exterior or interior of
$\Gamma$, respectively,
the following jump relations are well-known \citep{atkinson_1997,brebbia_1984,kress_1999}).

We assume $x\in\Gamma$ for the remainder of this section.
For the single-layer potential,
\begin{equation}
  S\sigma(x)
  =
  \lim_{x^\pm \rightarrow x}  S\sigma(x^\pm),
\end{equation}
for its derivative,
\begin{equation}
  \nabla_x S\sigma(x)
  =
  \lim_{x^\pm \rightarrow x}
  \nabla_{x^\pm} S\sigma(x^\pm)
  \pm \frac{1}{2} \hat n \, \sigma(x),
\end{equation}
where $\nabla_x S\sigma(x)$ is defined in the principal value sense, and
for the double-layer potential,
\begin{equation}
  D\mu(x)  = \lim_{x^\pm \rightarrow x}  D\mu(x^\pm) \mp \frac{1}{2} \mu(x).
  \label{eq:dlpjump}
\end{equation}

For higher derivatives with respect to the target location $x$, we have
\begin{align}
 \partial_{x_i}\partial_{x_j} S\sigma(x)
  &= \lim_{x^\pm \rightarrow x}
  \partial_{x^\pm_i}\partial_{x^\pm_j} S\sigma(x^\pm)
        \mp \frac \kappa2 (- \delta_{i,j} + 2\hat n_i \hat n_j)\sigma
        \pm \frac 12(\hat n_i \hat t_j + \hat t_j \hat n_i)
        \frac{d\sigma}{ds},
  \label{eq:slpderiv2}
\end{align}
where $\hat t$ is the unit tangent, assumed to satisfy
the identity
\[
\hat n = \begin{pmatrix} \phantom{-} \hat t_2 \\ -\hat t_1 \end{pmatrix}.
\]
$\kappa$ here is the curvature, $\delta$ is the Kronecker symbol and
$s$ is arc length.
The expression $\partial_{x_i}\partial_{x_j} S\sigma(x)$ is defined in the
Hadamard finite-part sense.

Finally,
in some settings, it is useful to consider derivatives of the double layer,
tangentially oriented dipoles, and mixed source/target derivatives. For these,
we have
\begin{align}
  \int_\Gamma v(x')\cdot \nabla_{x'} G(x,x') \sigma(x') \intd s
  &=
\lim_{x^\pm \rightarrow x}
  \left(\int_\Gamma v(x')\cdot \nabla_{x'} G(x^\pm,x') \sigma(x') \intd s\right)
  \mp \frac 12 (\hat n \cdot v(x))\sigma\\
  \partial_{x_i} D\sigma(x)
  &=
\lim_{x^\pm \rightarrow x}
  (\partial_{x^\pm_i} D\sigma(x^\pm))
  \mp \frac 12 \hat t_i \frac{d\sigma}{ds}(x)
\label{eq:dlpderiv}
\end{align}
In the former expression, $v$ is the direction in which the source
derivative is to be taken.
When it is tangentially oriented, $v(x') = \hat t(x')$, we denote the corresponding
operator by $R$:
\begin{equation}
 R\sigma(x) =  \int_\Gamma \hat t(x')\cdot \nabla_{x'} G(x,x') \sigma(x') \intd s.
\label{eq:riesz-def}
\end{equation}
The jump relations described in \eqref{eq:slpderiv2}-\eqref{eq:dlpderiv}
are not so well-known (see, for example, \citep{kolm_quadruple_2003}).

In summary, if the one-sided limit is the quantity of interest, then QBX computes
that directly and no post-processing work is required. If, however, the principal
value integrals $D\mu(x)$,
$\nabla_xS\sigma(x)$,
or the finite-part integrals in \eqref{eq:dlpderiv} are desired, then additional
steps are required.
The simplest scheme involves subtracting the relevant quantity from the QBX-derived
one-sided limit. This retains the expected order of accuracy.
A second option is to compute {\em both} one-sided limits using QBX and average the
quantities appropriately. That is, one can compute
\begin{equation}
D\mu(x) = \frac 12 \left( \lim_{x^+ \rightarrow x}  D\mu(x^+)
+  \lim_{x^- \rightarrow x}  D\mu(x^-) \right)
\label{eq:dlp-twosided}
\end{equation}
by two applications of QBX.

There are two drawbacks to the latter approach and one advantage.
First, it makes the scheme approximately twice as expensive as using a one-sided limit.
Second, the limiting values obtained from the two sides of $\Gamma$ can
vary noticeably in their accuracy for a given choice of smooth rule and expansion
order. As a result, the error in the
principal value computed by averaging is dominated by the worse of
the two limits. The advantage of using the two-sided limit is that
the Nyström approximation of the operator $D\mu(x)$ is much better behaved spectrally.
We discuss this issue in some detail in Section
\ref{sec:qbx-spectra}. When solving integral equations, we believe this advantage
outweighs the other considerations.

\subsection{Informal description of the algorithm}
\label{sec:at-a-glance}
This section provides a complete description of the steps required to
implement QBX. We assume that we are given a smooth curve $\Gamma$ subdivided
into $M$ panels $\Gamma_1,\dots,\Gamma_M$ of arc lengths
$h_1,\dots,h_M$, respectively and that $\hat n$ denotes the outward normal to
$\Gamma$.

\begin{quote}
\noindent
{\bf Set up parameters}
\begin{enumerate}
  \item
    Fix the desired accuracy $\epsilon$.

  \item Choose local expansion order $p$ (so that $S\sigma$ will be
    computed to the order of accuracy $p+1$).

  \item Choose $q$ and $r$ such that \eqref{eq:underlying-quad-estimate} is
    approximately satisfied to precision $\epsilon$ (assuming the underlying smooth
    rule is composite Gauss-Legendre quadrature).
    For points on panel $m$, a value of $r_m=h_m/2$ works well in practice.
    (See Section \ref{sec:analysis} for details.)
\end{enumerate}

\noindent
{\bf Compute one-sided limit}
\begin{enumerate}
  \setcounter{enumi}{3} 
  \item
    For each target point $x_j\in \Gamma_m \subset \Gamma$:
    \begin{enumerate}
      \item Fix the expansion center
        $c_j:= x_j \mp (h_m/2) \hat n,$ with $(-)$ corresponding to
        seeking the interior limit and $(+)$ corresponding to seeking
        the exterior limit.

        \smallskip
        \algcmt{%
          If $c_j$ is too close to any other panel $n\ne m$, refine
          the quadrature (``source") grid (thereby shrinking $h_m$ and
          moving $c_j$ closer to $\Gamma$) until this is no longer the case.%
        }
      \item Compute the expansion coefficients. For example, for the single layer potential,
        \[
          \alpha_{j,l}  :=
          \frac{i}{4} \,
          \int_\Gamma H^{(1)}_l(k|x'-c_j|) e^{i l \theta'} \sigma(x') \intd x'
        \]
        for $l=-p,\dots,p$ using the underlying $q$th order accurate rule.
        (See Fig. \ref{fig:graf-addition} for the definitions of
        $\theta,\theta'$.)
      \item Evaluate the local expansion at $c_j$:
        \[
          u_j:=
          \sum_{l=-p}^p \alpha_{j,l} J_l(k |x-c_j|) e^{-i l \theta}.
        \]
    \end{enumerate}
\item If the desired integral is a principal value or finite-part integral that
    has a jump condition, use the appropriate expression from Section
    \ref{sec:derivatives} to subtract the appropriate term from the one-sided
    limit (or repeat the calculation with a center on the opposite side
    and average the two sided limits as in
    \eqref{eq:dlp-twosided}).
\end{enumerate}
\end{quote}

A few observations are in order:

\begin{itemize}
\item In practice, the error from QBX is greater when the expansion center for
a target point $x$ lies on the concave side of the curve rather than the convex side.
(See Fig. \ref{fig:strip-of-death-coarse} for an illustration and
\citep{barnett_evaluation_2012} for analytic insight.)
\item
The algorithm contains a few nested loops, allowing for algorithmic
variation. One can save storage,
for example, by avoiding the allocation of memory to
the expansion coefficients $\alpha_{j,l}$. Each ``source" point can compute its
contribution to $u_j$ directly.
\item
The above algorithm directly \emph{applies} the layer
potential operator to a given density $\sigma(x')$. It is straightforward
to modify the algorithm to compute and store
all (or near neighbor) quadrature weights in a table, hence
constructing the Nyström matrix approximating the integral operator
(see Section \ref{sec:qbx-int-eq}).
\end{itemize}
\subsection{Remarks on grids}
\label{sec:grids}
The QBX procedure does not require tight coupling between the
``source" and ``target" grids. By ``source" grid, we mean the set of
points on $\Gamma$ where the density and Green's function are sampled
in computing the local expansion coefficients using the underlying
smooth quadrature rule. By ``target" grid, we mean the set of points
along $\Gamma$ where we seek the value of the layer potential.
\subsubsection{Adaptive boundary grids}
\label{sec:non-uniform-h}
\begin{figure}
  \centering
  \begin{tikzpicture}[scale=1.6]
    \coordinate (a) at (-3,0.5) ;
    \coordinate (b) at (1.5,0) ;
    \coordinate (c) at (2.75,0.55) ;
    \coordinate (d) at (2.9,1) ;
    \def\mycurve{
      \draw [|-|,thick]
        (a) ..controls +(20:1.5cm) and +(180:1.5cm) .. (b)
        \marklegendre{a} ;

      \draw [-|,thick]
        (b) ..controls +(0:4mm) and +(230:4mm) ..  (c)
        \marklegendre{b} ;

      \draw [-|,thick]
        (c) ..controls +(230+180:2mm) and +(280:2mm) ..  (d)
        \marklegendre{c} ;
    }
    \mycurve

    \foreach \intv/\intvh in {a/3,b/1.5,c/0.45}
    {
      \foreach \me/\other/\sign in
      {
        \intv 0/\intv1/1,
        \intv 1/\intv2/1,
        \intv 2/\intv3/1,
        \intv 3/\intv4/1,
        \intv 4/\intv5/1,
        \intv 5/\intv6/1,
        \intv 6/\intv5/-1
      }
      {
        \path
          let
            \p1 = ($ (\me) - (\other) $),
            \n2 = {veclen(\x1, \y1)},
            \p3 = (\y1*\sign/\n2*12*\intvh, -\x1*\sign/\n2*12*\intvh)
          in
            ($(\me) + (\p3)$) coordinate (c\me) ;
      }
    }

    \foreach \p/\q/\col in {
      a4/a5/black!50,
      a5/a6/black!40,
      a6/b0/black!30,
      b0/b1/black!30,
      b1/b2/black!40,
      b2/b3/black!50
    }
    {
      \fill [\col] (cb0) -- (\p) -- (\q) -- cycle;
    }
    \mycurve

    \foreach \intv in {a,b,c}
      \foreach \i in {0,...,6}
        \fill (\intv\i) circle (0.75pt);

    \foreach \intv in {a,b,c}
      \foreach \i in {0,...,6}
      {
        \draw (\intv\i) -- (c\intv\i) ;
        \fill (c\intv\i) circle (0.75pt);
      }

    \node [above=1.5cm of cb3] (centerlabel) {Expansion centers};
    \foreach \ctr in {ca3,ca5,cb3,cc3}
    {
        \draw [thin,->,shorten >=1mm] (centerlabel) -- (\ctr);
    }

    \node [above=5mm of cb1,xshift=2mm] (clabel) {$c$};
    \draw [->,shorten >=1mm] (clabel) -- (cb0);

    \node at (a1) [yshift=-5mm] {$\Gamma$} ;
    \node at (a3) [yshift=-3mm] {$I_1$} ;
    \node at (b4) [xshift=2mm,yshift=-3mm] {$I_2$} ;
    \node at (c4) [xshift=4mm,yshift=-2mm] {$I_3$} ;

  \end{tikzpicture}
  \caption{
    In some settings, one encounters source grids that are highly adaptive,
    with sudden changes in mesh spacing. This requires some control in QBX
    to avoid errors in the local expansion approximation
    (\ref{qbxquad-error}).
    See Section \ref{sec:non-uniform-h} for discussion.
  }
  \label{fig:unevenly-refined-panels}
\end{figure}
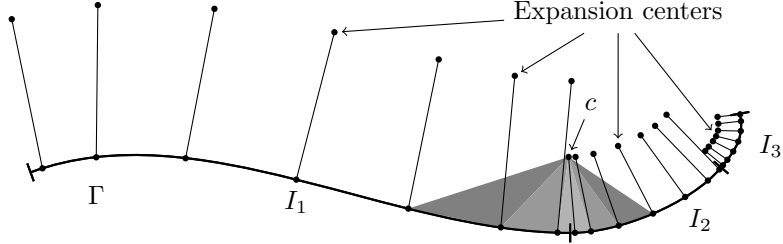

In our discussion thus far,
we have implicitly assumed the subintervals used to divide the boundary
$\Gamma$ are all of approximately the same length.
Many applications, of course, are best addressed using some form
of adaptive mesh refinement along the curve to resolve either complicated
data or to discretize an unknown but complicated density.

The main issue for the application of QBX  with such grids concerns
the location of the expansion centers and the validity of the error estimate
in \eqref{eq:underlying-quad-estimate}.
Fig. \ref{fig:unevenly-refined-panels}
illustrates the issue, under the assumption that centers are chosen
using the rule $r\approx h_1/2, h_2/2, h_3/2$ on three successive panels
$I_1,I_2,I_3$. Consider, now,
the expansion center indicated by $c$, which is at a distance
$r = h_2/2$ from $\Gamma$. The
filled triangles in the figure illustrate the angles spanned by adjacent
source quadrature nodes. The resolution
provided in evaluating expansion coefficients at $c$ by the sources on the
panel $I_1$ is clearly much lower than that provided by the sources on $I_2$ or $I_3$.
Moreover, the assumption that
$r > h_1/4$, which is  essential
in \eqref{eq:underlying-quad-estimate} in order
for the error to be small, is violated.

Fortunately, this problem is straightforward to address: one simply requires
sufficient sampling on $I_1$ for the error estimate to guarantee high precision.
There are several possible strategies in terms of implementation, and we list
two here.
\begin{itemize}
  \item If adjacent panels have substantially different lengths,
    interpolate the source density on the larger one to a fine grid that matches
    the resolution of the smaller one on the fly.
  \item In discretizing the boundary, require that no two adjacent panels differ
    in length by more than a factor of two and increase the number of points $q$
    per panel by a factor of 2.
\end{itemize}
We use the second (simpler) strategy for the data presented in Section
\ref{sec:numerical}.

\subsubsection{Grids for solving integral equations}
\label{sec:qbx-int-eq}
When solving integral equations with a Nyström method,
a common grid is used for both sampling the unknown density and
evaluating the resulting layer potential.
This coincides with what we have referred to as the target grid,
which should resolve the curve and the density to the desired
precision. This grid may not be sufficiently fine to
satisfy the requirements
\eqref{eq:underlying-quad-estimate} and \eqref{eq:x-c-closest}.
Under those conditions, in the QBX procedure, one simply needs to interpolate
the density to a finer grid, which becomes
what we have referred to as the source grid.
\subsection{Spectral structure of operators approximated by QBX}
\label{sec:qbx-spectra}
In this section, we consider the spectral structure of the
QBX-discretized layer potential operators. In addition to being of
mathematical interest, the spectral structure also plays an important
role in determining the
performance of iterative methods such as GMRES \citep{saad_gmres_1986}
when used to solve integral equations.

We concentrate here on the double layer potential $D\mu(x)$ which
plays an important role, for example, in
solving the Dirichlet problem for the Helmholtz equation
(at a non-resonant frequency $k$) in the interior $\Omega^-$ of $\Gamma$. Given
Dirichlet data $f(x)$, representing $\phi$ as a double layer potential
\[ \phi(x_0) = D\mu(x_0)  \]
for $x_0 \in \Omega^-$, and using the jump relation
\eqref{eq:dlpjump}, we obtain the equation
\begin{equation}
(-\frac{1}{2} + D ) \mu(x) = f(x)
\label{dir:skie}
\end{equation}
for $x \in \Gamma$.

On smooth boundaries, $D$ is continuous and hence compact,
with a discrete, bounded spectrum that has a unique accumulation point at zero
\citep{colton_inverse_1998}.
In other words, it is a smoothing operator that damps out the
high frequency modes in the density $\mu$.
We will show below that QBX is able to preserve all of these properties,
most critically the spectral clustering at zero.

We assume we have a grid on
$\Gamma$ with $M$ panels and $q$ points per panel, and that we consider
a density that lives in the space of piecewise $(q-1)$th order polynomials
over the $M$ panels.  We assume the double layer potential is computed
using QBX, with values output on the same grid, corresponding to a
discrete $Mq \times Mq$ matrix.

Now, because the truncated
$p$-term expansion represents locally smooth functions obeying a band
limit related to $p$, high-frequency components of the density
are either attenuated or aliased to lower frequencies as they
transition from the source grid to the expansion. Empirically,
attenuation is the dominant effect.

This behavior has several consequences. A beneficial
feature is that QBX responds very benignly to
potentially erroneous high-frequency data that may be present in the
discretized densities or geometries.  Also, some spectral features are
reproduced with no further effort.  For example, when applied to the
single layer potential, the QBX-based one-sided limit faithfully
reproduces the spectrum of the continuous operator, accumulating at
zero.

Unfortunately, when computing the \emph{double-layer potential} $D$
using QBX based on the one-sided limit as
\[
  D_{h,\text{one-sided}}\;\mu(x) = \lim_{x^\pm \rightarrow x}  D_h\mu(x^\pm) \mp \frac{1}{2} \mu(x) \, ,
\]
high frequency components are attenuated in
$\lim_{x^\pm \rightarrow x}  D_h\mu(x^\pm)$ but not, of course, in $\frac{1}{2} \mu(x)$.
As a result, the spectrum of $D_{h,\text{one-sided}}$ does {\em not} accumulate at zero.

If one then solves the integral equation \eqref{dir:skie} iteratively, with
$D_{h,\text{one-sided}}$ computed in this manner,
then an iterative method will converge rapidly
\emph{up to} the level of discretization error, at which point it will stall.
Since one does not know {\em a priori} exactly what the discretization error
will be, this is rather inconvenient.
Fortunately, computing $D$ using the
two-sided averaging approach \eqref{eq:dlp-twosided} as discussed
in Section \ref{sec:derivatives} yields a discrete
operator with a spectrum accumulating at zero, because both
limits are filtered. Matching this feature of the continuous operator
allows
iterative linear solvers converge to rapidly to solutions having residuals
near machine precision even \emph{beyond} the level of discretization error.
Using two-sided averaging as in \eqref{eq:dlp-twosided} may not be the
only way to achieve this spectral behavior, but it is particularly
convenient.

The preceding discussion applies only to the compact case.
For hypersingular (finite-part) integrals or Hilbert-Riesz type operators
such as $R$ in (\ref{eq:riesz-def}), there is no particular advantage in using the
two-sided limit.
Also, as discussed above, if jump conditions are not invoked, operators such as
the (compact) single-layer potential can be
represented faithfully by the one-sided procedure without difficulty.
In the numerical results shown in Section \ref{sec:numerical},
we have used two-sided averaging for all operators to which it
applies.
A more detailed discussion of the spectral properties of integral
operators computed using QBX will be reported at a later date.
\section{Numerical experiments}
\label{sec:numerical}
In this section, we illustrate the performance of QBX.
We begin by describing some simple test geometries.
We then present results for a variety
of layer potentials,
showing that high accuracy can be achieved even with
modest-sized discretizations.
Finally, we investigate the performance of QBX when used as part of an
integral equation solver for a variety of Dirichlet and Neumann
boundary value problems at various orders of accuracy.

For the sake of convenience, we will denote the normal
derivatives of the single and double layer potentials by
\begin{align*}
  S'\sigma(x) &:= \hat n(x) \cdot \nabla_x S\sigma (x),\\
  D'\sigma(x) &:= \hat n(x) \cdot \nabla_x D\sigma (x).
\end{align*}
When $x \in \Gamma$, the first is meant in the principal value sense
and the second in the Hadamard finite-part sense.
\subsection{Four test geometries}
\label{sec:test-geometries}
\begin{figure}
  \hfill
  \subfigure[
    A circle, decomposed into 50 panels, given by
    \eqref{eq:geo-ellipse} with $\alpha=1$.
  ]{
    \label{fig:geo-ellipse1}
    \includegraphics[width=0.4\textwidth]{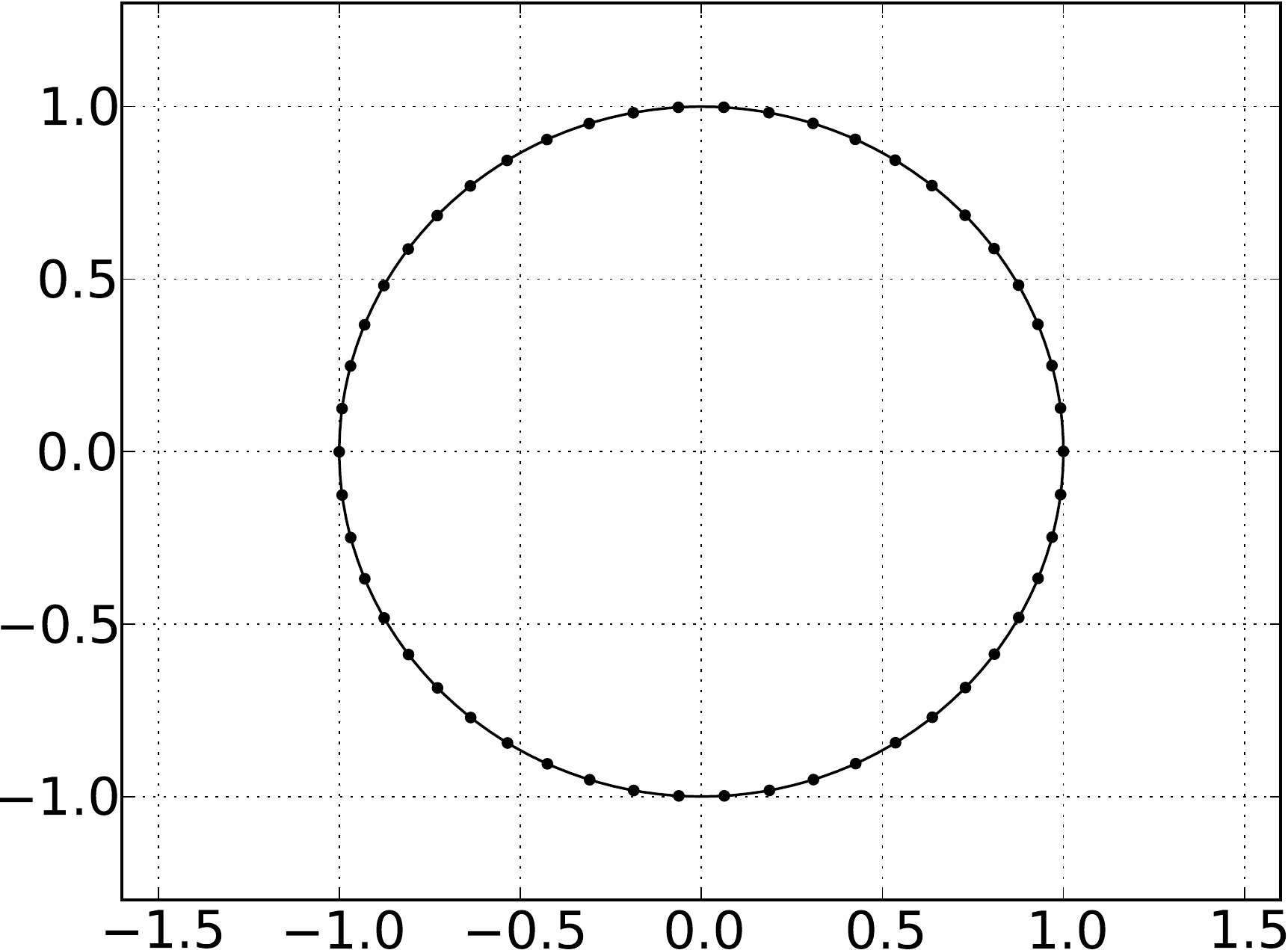}
  }
  \hfill
  \subfigure[
    An ellipse of aspect ratio 3:1 decomposed into 50 panels, given by
    \eqref{eq:geo-ellipse} with $\alpha=3$.
  ]{
    \label{fig:geo-ellipse3}
    \includegraphics[width=0.4\textwidth]{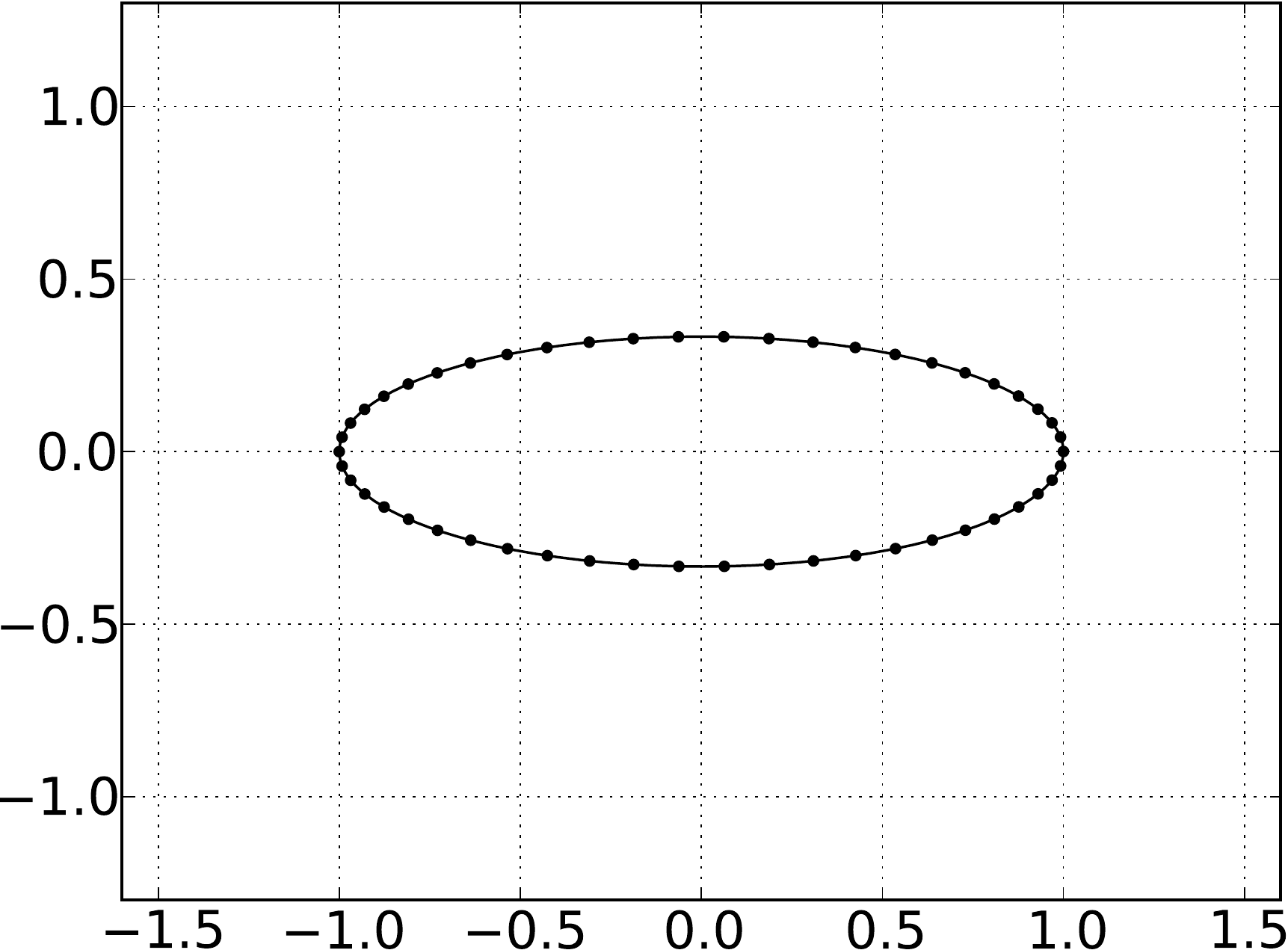}
  }
  \hfill

  \hfill
  \subfigure[An ellipse of aspect ratio 6:1 decomposed into 58 panels,
    given by \eqref{eq:geo-ellipse} with $\alpha=6$.
  ]{
    \label{fig:geo-ellipse6}
    \includegraphics[width=0.4\textwidth]{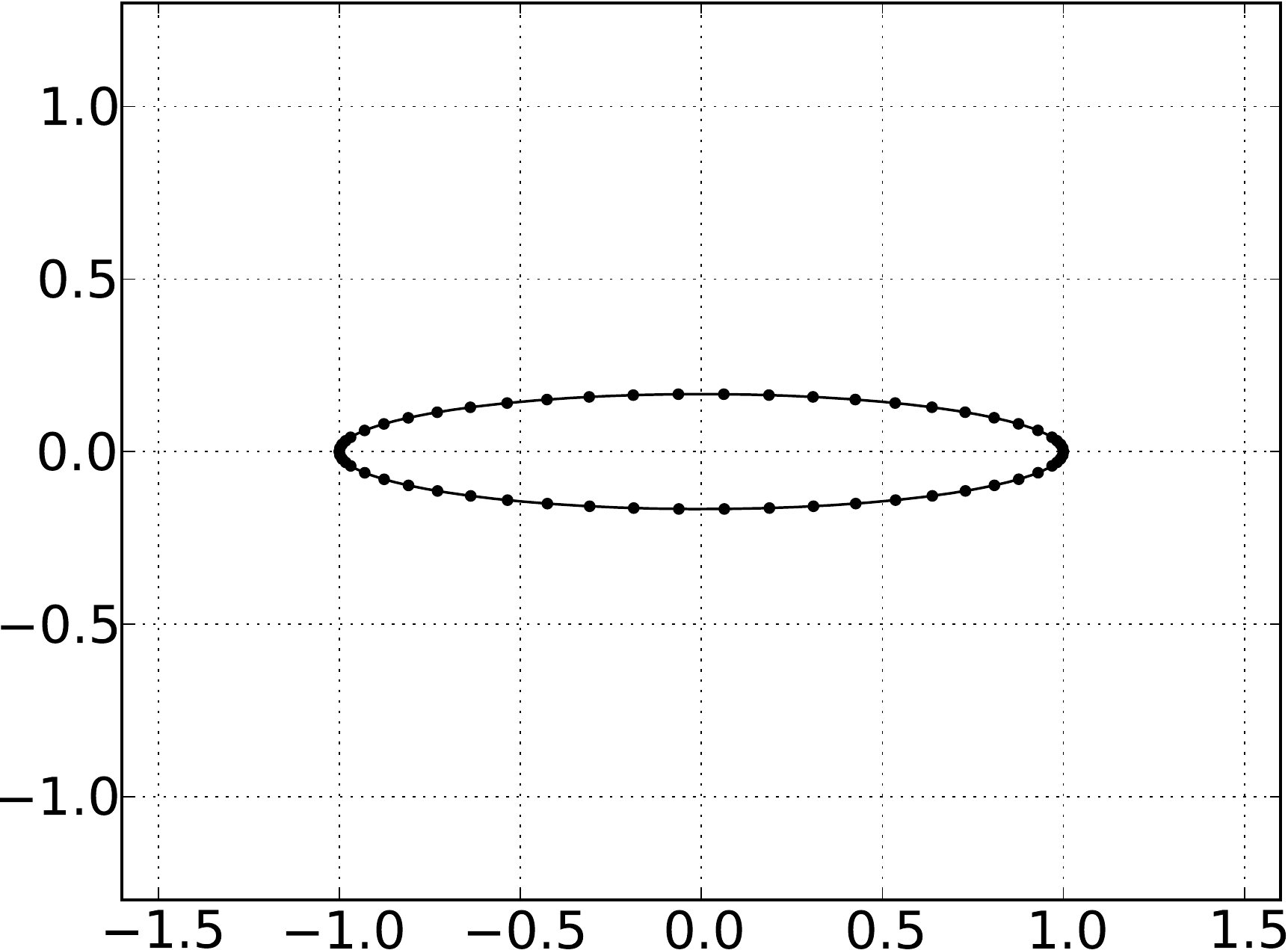}
  }
  \hfill
  \subfigure[
    A starfish-shaped curve decomposed into 80 panels,
    given by \eqref{eq:geo-starfish}.
  ]
  {
    \label{fig:geo-starfish}
    \includegraphics[width=0.4\textwidth]{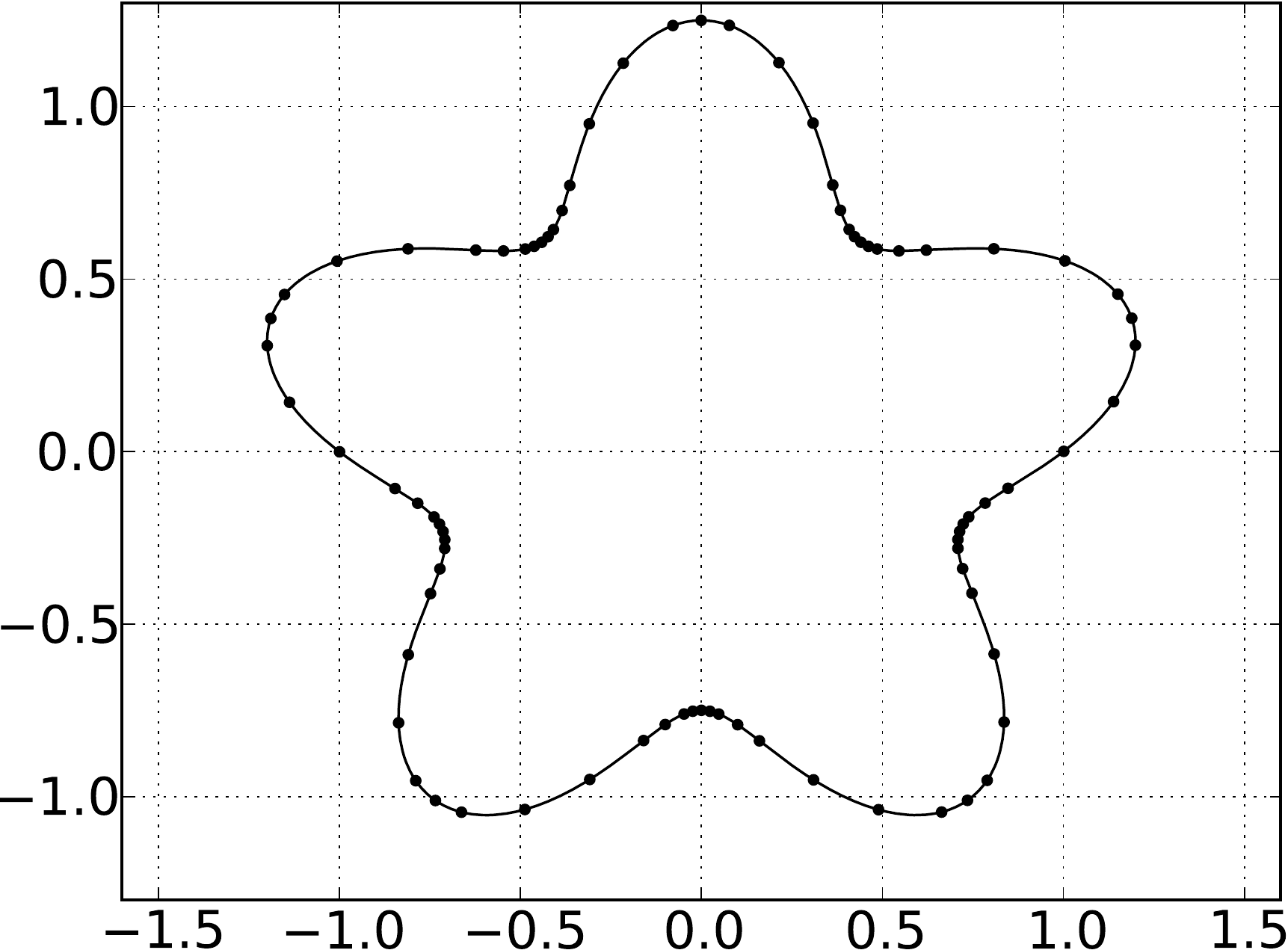}
  }
  \hfill

  \caption{Test geometries and their panel subdivisions.}
  \label{fig:test-geometry}
\end{figure}

The four curves that we will use for
our numerical tests are shown in Fig.
\ref{fig:test-geometry}.
The ellipses of Figs.
\ref{fig:geo-ellipse1}, \ref{fig:geo-ellipse3}, \ref{fig:geo-ellipse6}
are given by
\begin{equation}
  \gamma(t)=
  \begin{pmatrix}
    \phantom{\frac1\alpha}\cos(2\pi t) \\
    \frac1\alpha \sin(2\pi t)
  \end{pmatrix}
  \label{eq:geo-ellipse}
\end{equation}
for $\alpha=1$, $3$, and $6$, respectively, and the ``starfish" of Figure
\ref{fig:geo-starfish} is given by
\begin{equation}
  \gamma(t)=
  \left(1+\frac{\sin(5\cdot 2\pi t)}4 \right)
  \begin{pmatrix}
    \cos(2\pi t) \\
    \sin(2\pi t)
  \end{pmatrix}.
  \label{eq:geo-starfish}
\end{equation}
In each of these cases, $t\in[0,1)$.

We decompose the curves $\gamma(t) = (x(t),y(t))$ into panels, with
$x(t)$ and $y(t)$ represented by a 16-term Legendre polynomial
expansion on each panel.  We generate an initial subdivision that is
equispaced in $t$.  To ensure the accuracy of the expansion, we sample
the curve at $64=4\cdot 16$ points per panel and compute Legendre
expansion coefficients by Gauss-Legendre quadrature.

Since we are integrating with respect to the parameter $t$ rather
than arc length, we first determine whether the curve is well-resolved
by studying the spectral decay of the Legendre coefficients of
$|\gamma'(t)|=\sqrt{\gamma_1'(t)^2+\gamma_2'(t)^2}$,
using the method of \citet{kloeckner_viscous_2011}. We then
determine the $L^2$ energy contained in the tail of the series
(i.e. in modes 16 and above in our case). If the estimated
residual exceeds $10^{-11}$, the panel is bisected. A panel is also
bisected if its length (as computed by integrating $|\gamma'|$) is
more than twice that of its neighboring panels, to avoid the issues
described in Section \ref{sec:non-uniform-h}.

A source oversampling factor (see Section \ref{sec:grids}) of
$6$ is used throughout, that is $q=6 \cdot 16$.
A factor of $2$ is included to allow
adjacent panels to differ in length by factors of two, and an additional factor of $3$
is included to ensure that
the second term in the estimate \eqref{eq:underlying-quad-estimate} is
negligible.
More judicious oversampling strategies will be considered at a later date.
\subsection{Layer Potential Evaluation}
\label{sec:twoside-tests}
\begin{table}
  \centering
  \input{ref-digits.cls}

  \label{tab:ref-digits}
\end{table}

Our first set of tests examines the ability of QBX to compute a
range of standard and non-standard layer potential operators to high precision.
We consider the operators $S$,
$\partial_x S$, $\partial_y S$, $\partial_{xx} S$, $\partial_{xy} S$,
$\partial_{yy} S$ (target derivatives of the single-layer potential)
as well as $D$, $\partial_x D$, $\partial_y D$ (target derivatives of
the double-layer potential).  We also consider
the layer potential induced by tangentially
oriented dipoles (a source derivative in the tangential direction), which we denoted
earlier by $R$.
$R$ is the analog for Helmholtz potentials of the Hilbert transform in two dimensions
or the Riesz transform in three dimensions.

We apply each of these operators to the density
$\sigma(t)=\sin(10\pi t)$
and compare the computed result to a
reference solution in the $L^2$ and $L^\infty$ norms.  The Helmholtz
parameter was chosen as $k=0.5$.  The computations
were carried out with local expansion
order $p=16$.  We obtained our reference solution by using adaptive
Gaussian quadrature with tolerance $10^{-12}$ in quadruple precision with target points at
distances $10^{-6}$, $10^{-6}/2$, and $10^{-6}/4$ from the curve along the normal on
either side. We then computed one-sided limits $v^+$ and $v^-$ on each
side by third-order Richardson extrapolation.
We computed the value $(v^++v^-)/2$ as the reference solution
for principal value or finite-part on-surface integrals.
Results are shown in Table \ref{tab:ref-digits}, confirming
that high accuracy is achievable with
modest computational effort, as expected from a rapidly convergent scheme.
We further note that operators involving derivatives with tangential
components to the curve are either hypersingular or bounded (but not
compact).  Since differentiation is ill-conditioned, one should expect
some loss of accuracy with successively higher derivatives.

\subsection{Integral equation solvers}
\label{sec:inteq-solve-tests}
\begin{figure}
  \hfill
  \subfigure[
    Monopole ``point charges'' and observation points
    for the test of the solution of an
    exterior boundary value problem.
  ]{
    \includegraphics[height=5cm]{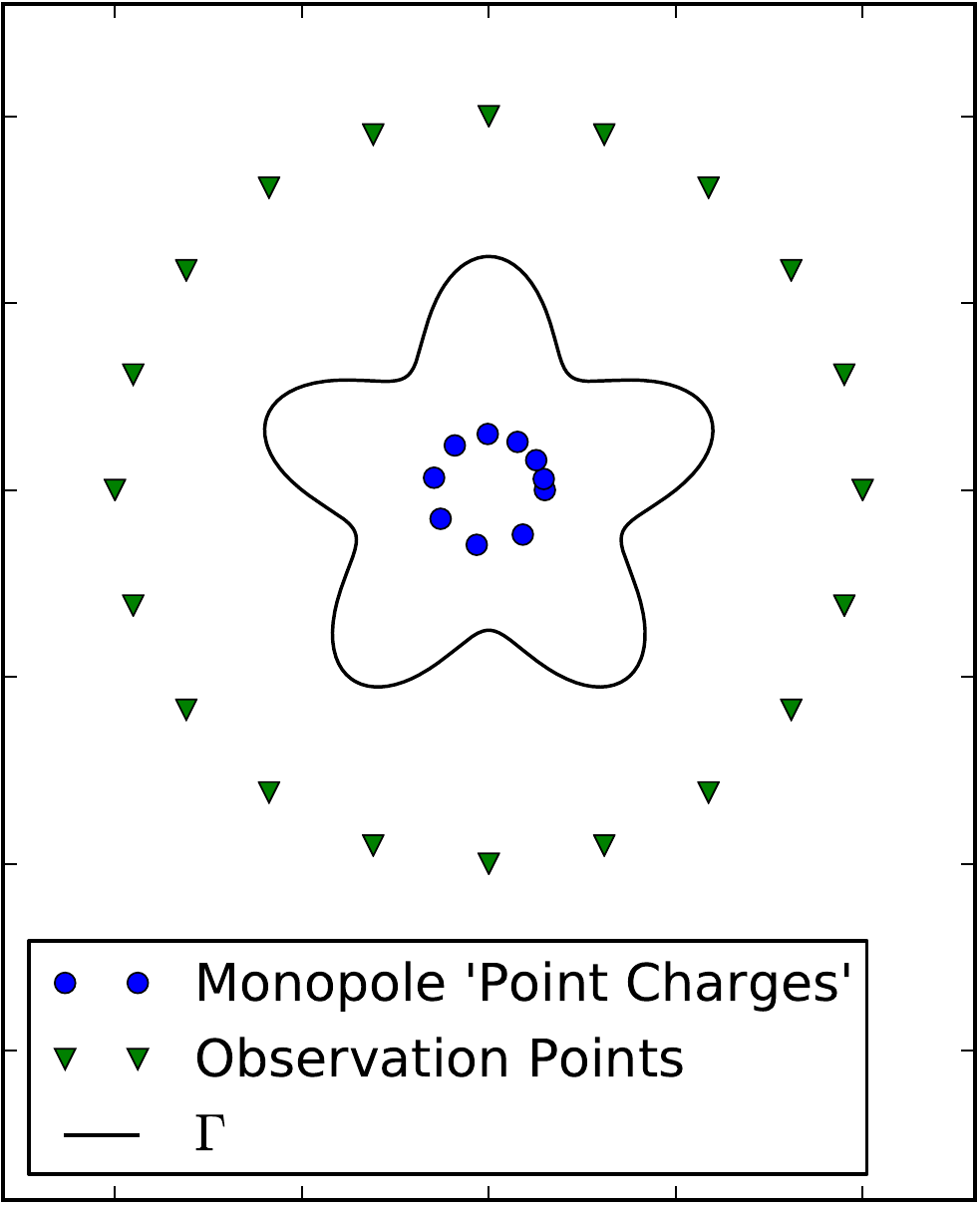}
  }
  \hfill
  \subfigure[
    Monopole ``point charges'' and observation points
    for the test of the solution of an
    interior boundary value problem.
  ]{
    \includegraphics[height=5cm]{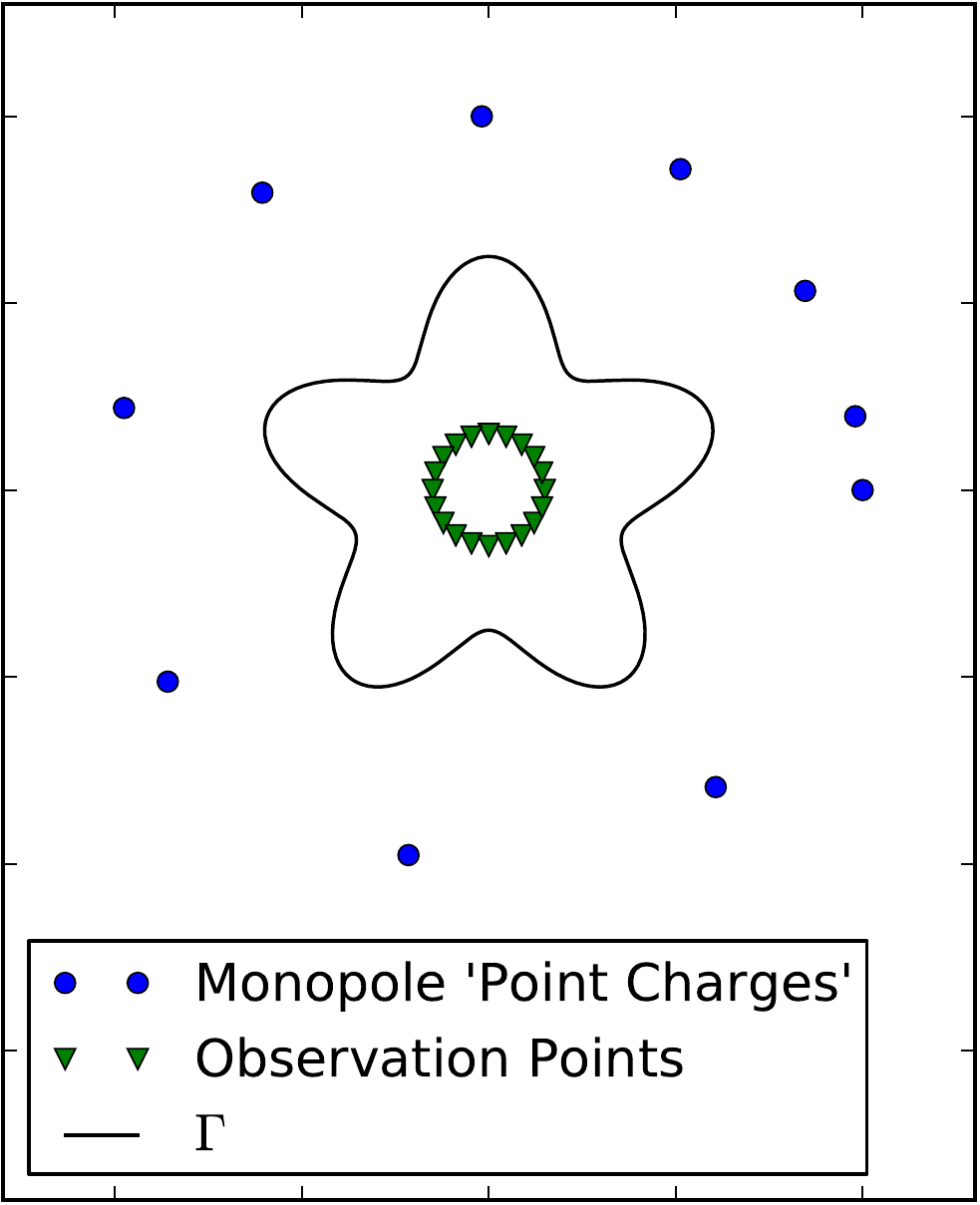}
  }
  \hfill 
  \ \caption{
    Setup of the integral equation test cases, shown with the
    `starfish' geometry of Figure \ref{fig:geo-starfish}.  The field
    induced by a collection of monopole ``point charges"
    in the complement of the computational domain is used to define
    the exact solution and to define the boundary condition for the
    governing partial differential equation. The relevant integral equation
    is then solved using QBX for discretization. Finally, the field is
    evaluated at the indicated observation points using the
    underlying smooth quadrature rule and compared to the reference
    field originating at the ``point charges.''
    (For observation points close to the boundary,
    the predecessor of QBX \citep{barnett_evaluation_2012} should be used.)
  }
  \label{fig:inteq-test-setup}
\end{figure}

\foreach \sym/\name in {
  ellipse1/circle,
  ellipse3/3-to-1 ellipse,
  ellipse6/6-to-1 ellipse,
  starfish/``starfish'' geometry
}
{
  \begin{table}
    \centering
    \input{inteq-\sym.cls}

    \caption{
      Convergence in the $l^2$ norm of the solution evaluated at a set of targets
      after solving a boundary value problem using an integral equation and QBX
      on the \name\ of Figure
      \ref{fig:geo-\sym}.
      GMRES iteration counts are shown in parentheses next to the
      error data.
      ``EOC'' is the empirical order of
      convergence, obtained by a log-least-squares fit of the shown
      $l^2$ errors.
    }
    \label{tab:inteq-results-\sym}
  \end{table}
}

\begin{figure}
  \hfill
  \subfigure[
    Discretization of the singular `teardrop' geometry, using 174
    panels of order 16. The panels near the corner are dyadically
    refined until the smallest one has length $10^{-8}$.
  ]{
    \label{fig:geo-drop}
    \includegraphics[width=0.35\textwidth]{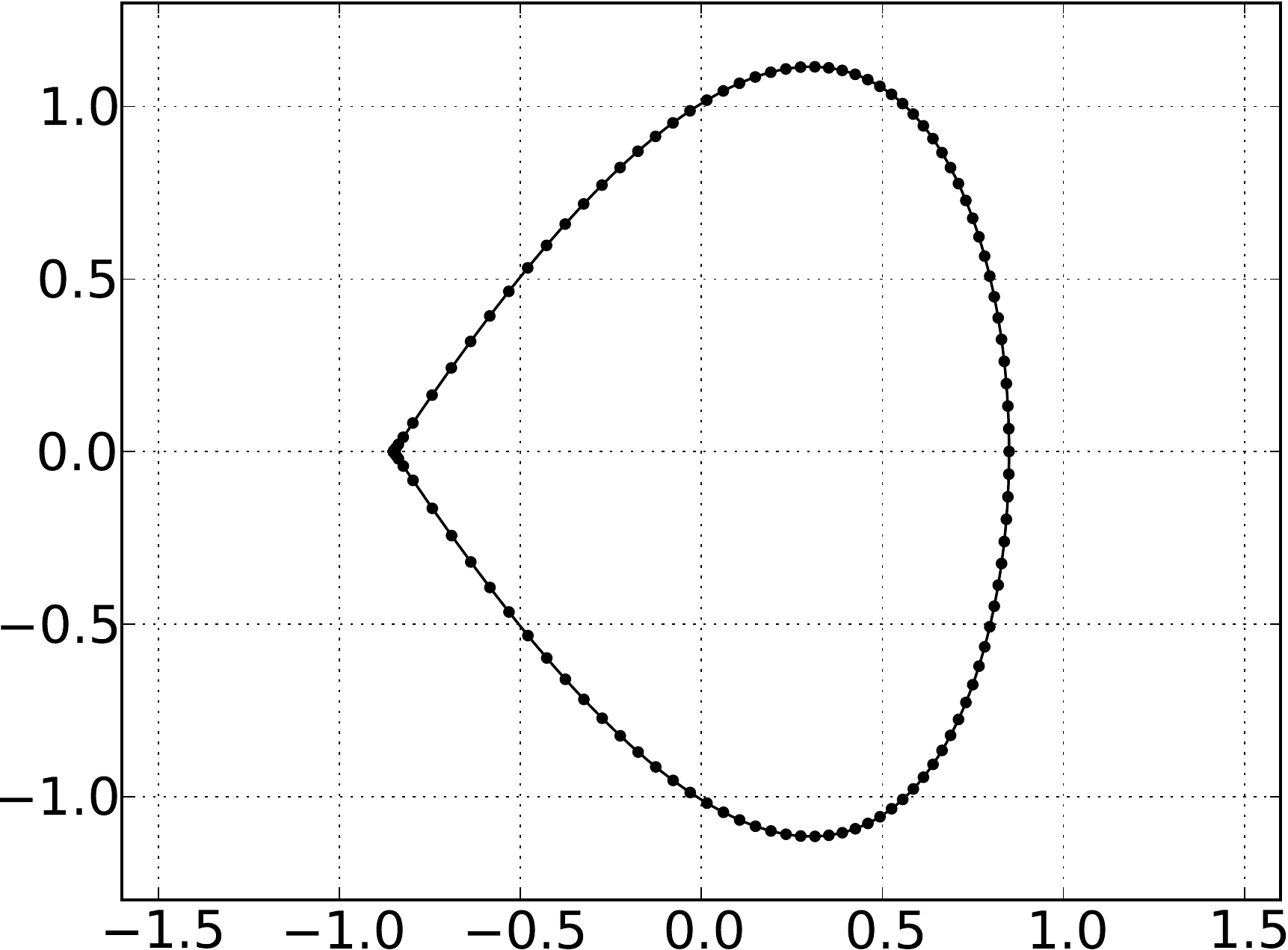}
  }
  \hfill
  \subfigure[
    Convergence in the $l^2$ norm of the solution evaluated at a set of targets
    after solving a boundary value problem using an integral equation and QBX
    on the `teardrop' geometry of Figure
    \ref{fig:geo-drop}.
  ]{\input{inteq-drop.cls}\label{tab:inteq-drop} }
  \hfill

  \caption{
    Integral equation tests on a `teardrop' geometry with a corner.
  }
\end{figure}

Our second and perhaps more important test examines the suitability of
QBX in the context of solving integral equations of the
second kind.
For each interior domain, we define an exact solution as the
field induced by a collection of monopole ``point charges'' in the exterior.
For each exterior domain, an exact solution is constructed using
monopole ``point charges'' in the interior. Given this exact solution, we
compute either Dirichlet or Neumann data and solve the corresponding
boundary value problem using an integral equation.
We then test the accuracy of the solution at a set of observation points.
Figure \ref{fig:inteq-test-setup} illustrates the geometry
of our tests for the interior and exterior case.

For the Dirichlet problem, we use the combined-field representation
$ u=  -D_k\sigma + \alpha S_k \sigma$ \citep{colton_inverse_1998}, which leads to the
second-kind equation
\begin{equation}
  \mp \frac 12 \sigma + \alpha S_k \sigma - D_k \sigma=f,
  \label{eq:dirichlet-inteq}
\end{equation}
for $\sigma$, where $f$ is the Dirichlet data obtained
from the exact solution.
Since the wave number is not large, we choose
$\alpha=i$ throughout this section.
The $(+)$ sign corresponds to the exterior problem
and the $(-)$ sign to the interior problem.
The subscript $k$ in $S_k$ and $D_k$ is used to emphasize that the underlying
Green's function is that for the Helmholtz equation with Helmholtz parameter $k$.

For the Neumann problem, we use a slight variation on the well-known
combined-field representation
\citep{panic_solubility_1965,leis_dirichletschen_1965,brakhage_uber_1965,colton_inverse_1998,bruno_fast_2012}
\[
   u= S_k\sigma  - \alpha D_k S_0 \sigma,
\]
where once again subscripts of $k$ indicate the use of the Helmholtz
kernel with parameter $k$, and a subscript of $0$ indicates the use of
a Laplace kernel.  This representation leads to the second-kind
integral equation
\begin{equation}
  \mp \frac 12 \sigma + S_k' \sigma - \alpha D_k' S_0 \sigma  = f
  \label{eq:neumann-inteq}
\end{equation}
for $\sigma$, where $f$ is the Neumann data obtained from the
exact solution.
Since QBX can integrate hypersingular kernels, we use \eqref{eq:neumann-inteq}
directly.
One may also use the Calder\'on projection identity
\begin{equation}
 D_0'S_0=-I/4+S_0'S_0'
 \label{eq:calderon}
\end{equation}
\cite{nedelec_acoustic_2001} and some algebra to avoid hypersingular operators.

Given our boundary discretization, we assume the unknowns are point values of
$\sigma$ at the source nodes and enforce the integral equation at the same nodes,
corresponding to a Nyström method. We use GMRES to solve
\eqref{eq:dirichlet-inteq} or \eqref{eq:neumann-inteq} iteratively and QBX to carry
out the matrix-vector products. We set the GMRES tolerance to $10^{-14}$ independent
of the order of accuracy of the QBX-based quadrature.

Following the work of \citet{bremer_nystrom_2011}, we use as unknowns the
density values multiplied by the square root of the corresponding quadrature weight.
This has the effect that the discrete $l^2$ inner product approximates the
continuous $L^2$ inner product and results in much improved conditioning,
especially in the presence of widely varying panel sizes. (This is critical
in geometries with corners, as discussed in the next section.)

After solving for $\sigma$ in \eqref{eq:dirichlet-inteq} or \eqref{eq:neumann-inteq},
we use the corresponding representation to evaluate the potential $u$
at a number of target points in $\Omega$. We then compare those values to
the exact solution and compute the relative
error.

Results for the geometries described in Section
\ref{sec:test-geometries} are shown in Tables
\ref{tab:inteq-results-ellipse1}, \ref{tab:inteq-results-ellipse3},
\ref{tab:inteq-results-ellipse6}, and
\ref{tab:inteq-results-starfish}. We observe that, while slightly more
erratic, the results for the Neumann operator
exhibit the loss of one order of accuracy, as expected since we have used QBX
for a hypersingular kernel.

Of particular note is the fact that, as
predicted in Section \ref{sec:qbx-spectra}, GMRES iteration
reaches a residual of $10^{-14}$ with a modest number of
iterations even for low order accurate discretization, in nearly all cases.
This makes QBX-based solvers particularly robust.
\subsubsection{Non-smooth geometries}
\label{sec:corners}
In the derivation of QBX, we have assumed that the layer potential is locally
smooth, so that an expansion in Bessel functions is rapidly convergent.
Since many engineering problems involve geometries with corners (and therefore
potentially non-smooth densities and unbounded layer potentials), it is of interest to study whether
QBX can be used effectively for such problems as well.

Without knowing the precise singularity in the density, it is
shown in \cite{bremer_nystrom_2011,helsing_corner_2008} that high-order polynomial approximation
combined with high-order quadrature on
a dyadically refined mesh yields high-order accuracy. Thus, the only question
is whether QBX can evaluate layer potentials on such structures without excessive
work.
To this end, we consider a `teardrop' shape with a single corner, described by the
parametrization
\[
  \gamma(t)=1.7\begin{pmatrix}
    \sin(\pi t)-0.5\\
    \frac12 \cos(\pi t)(\pi t-\pi)\pi t
  \end{pmatrix}.
\]
The curve and its discretization are shown in Figure \ref{fig:geo-drop}.
We have dyadically refined the boundary toward the corner until the smallest
panel lengths are less than $10^{-8}$ on each side.
Carrying out the same type of experiment as in the preceding section, we obtain
the results in Table \ref{tab:inteq-drop}. The apparent drop in
convergence order for $p=5$ can be attributed to the error made in
halting dyadic subdivision at $\epsilon=10^{-8}$.
These experiments
demonstrate that there are no significant obstacles to using QBX in this context.
Geometric panel refinements near corners or singularities mean that
QBX is still always evaluating a field locally smooth on the panel
scale.
\section{Generalizations and implementation issues}
\label{sec:details}
Our goal in this paper has been to present a new approach to quadrature, which we refer to as QBX
(`quadrature by expansion'). While we have largely limited
our attention to the Helmholtz equation in two dimensions, it should be
clear that the overall approach is independent of dimension
as well as the precise nature of the governing Green's function.
In fact, the QBX approach is far \emph{more} general, extending to kernels that
are not directly connected to a partial differential equation.
These extensions are discussed in \citep{kloeckner_general_2012}.
The method gives rise to
a number of important and interesting questions regarding efficiency,
robustness, and automatic adaptivity.

As for implementation, the main issue we have ignored here is
computational cost. As presented in Section \ref{sec:at-a-glance}, the
asymptotic complexity of QBX is $O(NN_t)$, where $N$ is the number of
source points, $N_t$ is the number of target points. Neglecting
numerous opportunities for optimization, a straight implementation of
the algorithm of Section \ref{sec:at-a-glance} can apply a
single-layer operator at order $p=5$ to a density on 7680 source nodes
with 1280 targets in 0.6 seconds using 16 cores of a 2.93 GHz Intel Nehalem machine. Fortunately,
QBX can be accelerated using the fast multipole method (FMM) or any
other hierarchical fast algorithm
\citep{greengard_fast_1987,cheng_remarks_2006}. The cost is then $O(N
\log N + N_t \log N_t)$.  This coupling, further cost savings, as well as extensions to
three dimensions, are discussed in \citep{qbx_acceleration_2012}.  To
give an indication of the achievable speedups, preliminary
implementations show that the cost of an FMM-based QBX scheme for a
layer potential is two to three times that for a point-to-point FMM
procedure. In particular, layer potentials with tens of thousands of
discretization points are computed in seconds on a single CPU core.  A
variety of other optimizations are also possible: using direct
evaluation for distant panel interactions and QBX for near neighbors
only, adaptive oversampling to ensure accuracy of the local expansion
coefficients with highly adaptive and irregular panel sizes, using the
sample local expansion for several nearby target points, etc.

\section{Conclusions}
\label{sec:conclusions}
QBX permits the rapid, high-order accurate
evaluation of layer potentials in a manner that
is remarkably easy to implement. It is based
on the fact that the induced potential is smooth in the exterior or interior
domain.
The scheme is equipped with a
complete convergence theory. With minor modifications, QBX can
evaluate layer potentials at off-surface points arbitrarily close to the boundary. Since
it is an extension of the scheme developed in \citep{barnett_evaluation_2012} for
precisely that purpose, this is not a surprise.
QBX presents an opportunity
to develop a robust set of software tools for evaluating integral operators with
singular or weakly singular kernels, with
application to a broad range of large-scale simulations in physics and engineering.


\section*{Acknowledgments}

The authors would like to thank Z.~Gimbutas, C.~Epstein, J.-Y.~Lee,
S.~Jiang, S.~Veerapaneni, M.~Tygert, and T.~Warburton for fruitful
discussions.  AK would also like to acknowledge the use of computing
resources supplied by T.~Warburton. The authors' work was supported
through the AFOSR/NSSEFF Program Award FA9550-10-1-0180, by NSF grant
DMS-0811005, and by the Department of Energy under contract
DEFG0288ER25053.

\bibliography{qbx}
\bibliographystyle{abbrvnat}
\end{document}